\documentclass[leqno,english]{amsart}
\usepackage[T1]{fontenc}
\usepackage[utf8]{inputenc}
\usepackage{xcolor}
\usepackage{babel}
\usepackage{enumitem}
\usepackage{amstext}
\usepackage{amsthm}
\usepackage{amssymb}
\usepackage{setspace}
\usepackage[pdfusetitle,
 bookmarks=true,bookmarksnumbered=true,bookmarksopen=false,
 breaklinks=false,pdfborder={0 0 0},pdfborderstyle={},backref=false,colorlinks=true]
 {hyperref}
\hypersetup{
 citecolor=blue,linkcolor=red}

\makeatletter
\numberwithin{equation}{section}
\numberwithin{figure}{section}
\numberwithin{table}{section}

\usepackage{tikz,fancyhdr,enumitem}
\usetikzlibrary{scopes,calc,intersections,through,backgrounds,arrows,decorations.pathmorphing,positioning,fit,petri,decorations.markings,matrix}
\usepackage{hyperref,bookmark,dsfont,bbm}
\usepackage[all]{xy}
\usepackage{microtype}
\hypersetup{linktocpage}


\makeatother

\providecommand\theoremname{Theorem}
\theoremstyle{plain}
\newtheorem{thm}{\protect\theoremname}[section]
\providecommand\examplename{Example}
\theoremstyle{definition}
\newtheorem{example}[thm]{\protect\examplename}
\providecommand\remarkname{Remark}
\theoremstyle{remark}
\newtheorem{rem}[thm]{\protect\remarkname}
\providecommand\propositionname{Proposition}
\theoremstyle{plain}
\newtheorem{prop}[thm]{\protect\propositionname}
\providecommand\definitionname{Definition}
\theoremstyle{definition}
\newtheorem{defn}[thm]{\protect\definitionname}
\providecommand\lemmaname{Lemma}
\theoremstyle{plain}
\newtheorem{lem}[thm]{\protect\lemmaname}
\providecommand\corollaryname{Corollary}
\newtheorem{cor}[thm]{\protect\corollaryname}
\providecommand\problemname{Problem}
\theoremstyle{definition}
\newtheorem{problem}[thm]{\protect\problemname}
\begin{document}
\title{Syzygy and Adjoint Matroids}
\author{Jaeho Shin}
\address{Department of Mathematical Sciences, Seoul National University, Seoul
08826, South Korea}
\email{j.shin@snu.ac.kr}
\begin{abstract}
We give a matroid-theoretic analogue of Gaussian elimination and use
it to construct, for every matroid $M$, the syzygy matroid $\mathrm{Syz}(M)$
on the collection of circuits of $M$. Dualizing the construction
yields the combinatorial adjoint $\mathrm{Adj}(M)$ on the collection
of maximal proper flats of $M$, and the two are related by $\mathrm{Adj}(M)\simeq\mathrm{Syz}(M^{\ast})$.
Neither construction depends on a realization. We determine $\mathrm{Syz}(M)$
when $M$ is $U_{k,k+3}$ or the Vámos matroid $V_{8}$. We also show
that $\mathrm{Adj}(M)$ is an adjoint of $M$ when $M$ has rank $3$.
When $M$ has rank $4$ and $V_{8}$ is a restriction of $M$, $\mathrm{Adj}(M)$
has rank $3$ and is not an adjoint of $M$.
\end{abstract}

\maketitle

\section{\label{sec:Introduction}Introduction}

A matroid abstracts linear dependence, and one may ask, as Rota and
Crapo did \cite{Kung}, whether it gives rise to a second matroid
built from its own structure. For realizable matroids there are two
candidates, one from the circuits and the other from the maximal proper
flats. The former goes back to Longyear \cite{Longyear} for binary
matroids and was defined for realizable matroids by Oxley and Wang
\cite{Oxley 2}; the latter has appeared as the type I adjoint of
Bixby and Coullard \cite{BC88}, as the matroid of the derived code
of Jurrius and Pellikaan \cite{JP15}, and as the cocircuit matroid
described by Kung, following Crapo \cite{Kung}. The two constructions
are related by matroid duality, and neither is determined by the matroid
alone.

The classical approach to building a matroid on the maximal proper
flats without a realization is the adjoint of Cheung \cite{Cheung},
a geometric lattice of the same rank into which the inverted flat
lattice of a matroid $M$ embeds with atoms mapped onto atoms. Yet
an adjoint need not exist: the Vámos matroid $V_{8}$ has none \cite{Cheung}.
Another approach employs the Dilworth completion, which embeds an
arbitrary finite lattice into a geometric lattice. Applied to the
inverted flat lattice of $M$, it produces a geometric lattice, called
the \emph{complete adjoint} of $M$ \cite[Exercise 7.59]{Dilworth}.
However, the embedding does not control the ranks of the images of
the flats of $M$, so the complete adjoint is not an adjoint in general,
as $V_{8}$ already shows.

In this paper, for every matroid $M$, we construct two matroids determined
by $M$ alone: the \emph{syzygy matroid} $\mathrm{Syz}(M)$ on the
collection of circuits of $M$ and the \emph{combinatorial adjoint}
$\mathrm{Adj}(M)$ on the collection of maximal proper flats of $M$.
As in the realizable case, the two are related by the duality $\mathrm{Adj}(M)\simeq\mathrm{Syz}(M^{*})$.

The construction of $\mathrm{Syz}(M)$ proceeds in two stages. The
first stage takes an arbitrary finite set as the ground set and extends
any \emph{seed}, that is, any collection of nonempty subsets of that
set, to the dependent-set collection of a matroid by adding the sets
required by the elimination axiom for dependent sets, taking upward
closure, and repeating these two steps until the collection stabilizes.
This stage requires no matroid at the outset, in contrast to single-element
extensions, which are governed by the modular cuts of a given matroid
\cite{Crapo}. 

The second stage uses an analogue of Gaussian elimination to measure
how many circuits are needed to span a given collection of circuits.
The matroid $\mathrm{Syz}(M)$ is obtained by applying the first stage
to the seed consisting of those collections of circuits of $M$ that,
in this sense, are spanned by fewer circuits than they contain.\footnote{The relations among a set of generators of a module form a first syzygy
module of that module. If a matroid $M$ is realized over a field
$\mathbb{F}$ by $L:\mathbb{F}^{n}\rightarrow\mathbb{F}^{k}$, the
relations among the columns of $L$ form $\ker L$, whose nonzero
elements of minimal support are the circuit vectors of Section \ref{sec:Realizations};
hence the name \emph{syzygy matroid}. Dependencies among circuits
thus play the role of syzygies among syzygies, the analogy behind
Rota's program \cite{Kung}.} The first stage also yields the combinatorial derived matroid of
Freij-Hollanti, Jurrius, and Kuznetsova \cite{CDM} when applied to
a seed included in this one. However, the first stage is not monotone
in its input, so this inclusion implies no comparison between the
two matroids. 

Dualizing the construction, that is, applying it to $M^{\ast}$ and
taking complements in $E(M)$, yields the combinatorial adjoint $\mathrm{Adj}(M)$.
Under this complementation, dependent sets of $M^{\ast}$ become nonspanning
sets of $M$. The matroid $\mathrm{Adj}(M)$ can be computed either
from $\mathrm{Syz}(M^{\ast})$ or directly within the flat lattice
of $M$. Like $\mathrm{Adj}(M)$, the complete adjoint exists for
any matroid, but the two do not coincide in general; see Remark~\ref{rem:Dilworth-Adjoint}. 

Both $\mathrm{Syz}(M)$ and $\mathrm{Adj}(M)$ are simple, and both
are connected when $M$ is. We determine $\mathrm{Syz}(M)$ for $M=U_{k,k+3}$
and $M=V_{8}$. If $M$ has rank $3$, then $\mathrm{Adj}(M)$ is
an adjoint of $M$. In rank $4$ this can fail: whenever $M$ has
$V_{8}$ as a restriction, $\mathrm{Adj}(M)$ has rank $3$ and is
therefore not an adjoint of $M$. 

The paper is organized as follows. Section \ref{sec:Realizations}
describes the two constructions in the realizable case and their duality,
and Section \ref{sec:Generating-Matroids} shows how any collection
of nonempty subsets of an arbitrary finite set generates a matroid.
The analogue of Gaussian elimination is developed in Section \ref{sec:Syzygy},
where $\mathrm{Syz}(M)$ is constructed, with loops split off as free
summands. Section \ref{sec:Combinatorial-Adjoints} transports the
construction to the maximal proper flats and studies $\mathrm{Adj}(M)$,
with coloops split off dually.

\subsection{\label{subsec:Notation-terminology}Notation and terminology }

We write $E(M)$ for the ground set of a matroid $M$. We let $n$
and $k$ denote $\left|E(M)\right|$ and $r(M)$, respectively, and
identify $E(M)$ with $\left[n\right]=\left\{ 1,\dots,n\right\} $
when the elements need to be indexed, in which case we may write a
subset by concatenating its elements, as in $134=\left\{ 1,3,4\right\} $.
We denote by $\mathcal{C}(M)$, $\mathcal{B}(M)$, and $\mathcal{L}(M)$
the collections of circuits, bases, and flats of $M$, respectively.
A subscript $M$ on a function or operator indicates that it is taken
with respect to $M$, as in the rank function $r_{M}$ and the closure
operator $\mathrm{cl}_{M}$. We omit the subscript when the matroid
is clear from the context.

Throughout, \emph{contains} refers to membership and \emph{includes}
to inclusion: a set contains its elements and includes its subsets.
\emph{Hyperplane} is reserved for its classical meaning, a codimension-one
subspace of a vector or projective space. A flat that is maximal among
the proper flats of $M$ is called a \emph{maximal proper flat}, or
\emph{mp-flat} for short; such a flat is commonly called a hyperplane
in the matroid literature. We write $\mathcal{H}(M)$ for the collection
of mp-flats of $M$. For a subset $X\subseteq E(M)$, the \emph{corank}
of $X$ in $M$ is $\mathrm{corank}_{M}(X):=r(M)-r_{M}(X)$, so the
mp-flats of $M$ are precisely its flats of corank $1$. This quantity
should not be confused with the rank $r_{M^{\ast}}(X)$ of $X$ in
$M^{\ast}$, which is sometimes called the corank in the literature. 

\section{\label{sec:Realizations}Constructions from Realizations}

A realization of a matroid yields one matroid from the circuit vectors
and another from the normal vectors of the hyperplanes corresponding
to the mp-flats. Both depend on the realization, whereas the constructions
of Sections \ref{sec:Syzygy} and \ref{sec:Combinatorial-Adjoints}
do not.\smallskip{}

Let $L$ be a $k\times n$ matrix of rank $k$ over a field $\mathbb{F}$,
let $V$ be its row space, and let $V^{\perp}$ be the orthogonal
complement of $V$ in $\mathbb{F}^{n}$. For $e\in\left[n\right]$
and $X\subseteq\left[n\right]$, denote by $L_{e}$ the column of
$L$ indexed by $e$, and by $L_{X}$ the submatrix of $L$ consisting
of the columns indexed by $X$. We read $L$ in two complementary
ways. 
\begin{enumerate}[label=(LS)]
\item \textbf{\label{enu:Span-viewpoint}}\textbf{Linear span viewpoint.}
Regard the columns of $L$ as vectors in $\mathbb{F}^{k}$. The sets
$X\subseteq[n]$ such that the columns of $L_{X}$ are linearly independent
form the independent sets of a matroid $M$ on $[n]$ of rank $k$.
When we say that a collection of vectors forms a matroid, we mean
it in this sense. The zero columns of $L$ are the loops of $M$,
and each nonzero column defines a point in $\mathbb{P}^{k-1}=\mathbb{P}(\mathbb{F}^{k})$;
thus $M$ records the linear dependencies among these points.
\end{enumerate}
\begin{enumerate}[label=(HA)]
\item \textbf{\label{enu:HA-viewpoint}}\textbf{Hyperplane arrangement
viewpoint.} The projective space $\mathbb{P}(V)$ can be identified
with $\mathbb{P}^{k-1}$ through the isomorphism induced by $\mathbb{F}^{k}\to V$,
$\mathbf{x}\mapsto\mathbf{x}^{T}L$. Regard the columns of $L$ as
linear functionals on $V$. These functionals satisfy the same linear
relations as the columns, and hence also form the matroid $M$. Each
nonzero functional defines a hyperplane in $\mathbb{P}(V)$, and $M$
records the incidence relations among these hyperplanes. 
\end{enumerate}

\subsection{\label{subsec:OW}Derived matroids}

The \emph{support} of a vector $\mathbf{v}\in\mathbb{F}^{n}$ is the
set of coordinates at which $\mathbf{v}$ is nonzero, namely $\mathrm{supp}(\mathbf{v})=\left\{ e\in\left[n\right]:\mathbf{v}(e)\neq0\right\} $.
A vector $\mathbf{v}\in\mathbb{F}^{n}$ lies in $V^{\perp}$ exactly
when $\sum_{e\in[n]}\mathbf{v}(e)L_{e}=\mathbf{0}$, so a nonzero
$\mathbf{v}\in V^{\perp}$ records a linear relation among the columns
of $L_{\mathrm{supp}(\mathbf{v})}$, and $\mathrm{supp}(\mathbf{v})$
is a dependent set. Conversely, every dependent set includes $\mathrm{supp}(\mathbf{v})$
for some nonzero $\mathbf{v}\in V^{\perp}$. Applied to a circuit
$C$, this yields a nonzero $\mathbf{v}\in V^{\perp}$ with $\mathrm{supp}(\mathbf{v})=C$,
since $\mathrm{supp}(\mathbf{v})$ is dependent and no proper subset
of $C$ is; such a vector is called a \emph{circuit vector} of $C$.
If $\mathbf{v}_{1}$ and $\mathbf{v}_{2}$ are circuit vectors of
$C$ and $e\in C$, then $\mathbf{v}_{1}(e)$ and $\mathbf{v}_{2}(e)$
are nonzero, and the vector $\mathbf{w}=\mathbf{v}_{2}(e)\mathbf{v}_{1}-\mathbf{v}_{1}(e)\mathbf{v}_{2}\in V^{\perp}$
vanishes at $e$. Its support is thus included in the independent
set $C\backslash\{e\}$, so $\mathbf{w}=\mathbf{0}$. Hence a circuit
vector of $C$ is unique up to a nonzero scalar multiple, and we speak
of \emph{the} circuit vector of $C$. 

Let $B$ be a base of $M$. For each $e\in[n]\backslash B$, there
exists a unique circuit included in $B\cup\{e\}$, called the \emph{fundamental
circuit} of $e$ with respect to $B$ and denoted by $C(e,B)$. Note
that $e\in C(e,B)$. For $j\in B\cup\{e\}$, the set $(B\cup\{e\})\backslash\{j\}$
is a base of $M$ if and only if $j\in C(e,B)$. There are exactly
$n-k$ fundamental circuits with respect to $B$.

We now describe circuit vectors explicitly. Since every circuit is
a fundamental circuit with respect to some base and $B$ is arbitrary,
it suffices to do so for $C=C(e,B)$ with $e\in\left[n\right]\backslash B$.
The submatrix $L_{B}$ is an invertible $k\times k$ matrix, so $\tilde{L}:=L^{-1}_{B}L$
is a $k\times n$ matrix with $\tilde{L}_{B}=I_{k}$, and the column
$\tilde{L}_{e}$ can be written uniquely as $\tilde{L}_{e}=\sum_{j\in B}a_{j}\tilde{L}_{j}$
with $a_{j}\in\mathbb{F}$. Set $a_{e}=-1$ and $a_{j}=0$ for $j\notin B\cup\{e\}$,
and define $\mathbf{v}_{C}\in\mathbb{F}^{n}$ by $\mathbf{v}_{C}(j)=a_{j}$.
If $\mathbf{v}$ is the circuit vector of $C$ scaled so that $\mathbf{v}(e)=-1$,
then $\sum_{j\in[n]}\mathbf{v}(j)\tilde{L}_{j}=\mathbf{0}$ and $\mathrm{supp}(\mathbf{v})\subseteq B\cup\{e\}$,
so $\tilde{L}_{e}=\sum_{j\in B}\mathbf{v}(j)\tilde{L}_{j}$. By the
uniqueness of this expression, $\mathbf{v}=\mathbf{v}_{C}$. Thus
$\mathbf{v}_{C}$ is the circuit vector of $C$, normalized so that
$\mathbf{v}_{C}(e)=-1$. 

For distinct $e,e'\in[n]\backslash B$, we have $\mathbf{v}_{C(e,B)}(e)=-1$
and $\mathbf{v}_{C(e,B)}(e')=0$, so the circuit vectors of the $n-k$
fundamental circuits with respect to $B$ are linearly independent.
Since $\dim V^{\perp}=n-k$, they span $V^{\perp}$. Therefore the
circuit vectors $\mathbf{v}_{C}$ with $C\in\mathcal{C}(M)$ span
$V^{\perp}$, and they form a matroid on $\mathcal{C}(M)$ of rank
$n-k$. This matroid, denoted by $\delta_{OW}(M,L)$, is the \emph{derived
matroid} of $(M,L)$ introduced by Oxley and Wang \cite{Oxley 2}. 

\subsection{Adjoint matroids}

Let $J\subseteq\left[n\right]$ be a subset of size $k-1$. Let $L_{J,i}$
denote the matrix obtained from $L_{J}$ by deleting the $i$-th row.
Define a vector $\mathbf{c}_{J}\in\mathbb{F}^{k}$ by 
\[
\mathbf{c}_{J}(i)=\left(-1\right)^{i}\det L_{J,i}.
\]
 The vector $\mathbf{c}_{J}$ is nonzero if and only if the rank of
$L_{J}$ is $k-1$. In this case, $\mathbf{c}_{J}$ is orthogonal
to the hyperplane spanned by the columns of $L_{J}$.

Two nonzero vectors $\mathbf{c}_{J},\mathbf{c}_{J'}\in\mathbb{F}^{k}$
are parallel if and only if they determine the same hyperplane, equivalently,
if and only if $\mathrm{cl}_{M}(J)=\mathrm{cl}_{M}(J')$. In this
case, they represent the same point in $\mathbb{P}^{k-1}$. If $H$
is an mp-flat of $M$, then the vectors $\mathbf{c}_{J}$ for the
independent sets $J\subseteq H$ with $\left|J\right|=k-1$ all represent
the same point in $\mathbb{P}^{k-1}$, which we denote by $\mathbf{c}_{H}$.
We use $\mathbf{c}_{H}$ to denote both the point in $\mathbb{P}^{k-1}$
and any representative vector in $\mathbb{F}^{k}$; the intended meaning
will be clear from the context. From the hyperplane arrangement viewpoint
\ref{enu:HA-viewpoint}, $\mathbf{c}_{H}$ is the point of $\mathbb{P}^{k-1}$
at which the hyperplanes defined by the nonzero columns of $L_{H}$
meet. 

It follows that $\mathbf{c}_{H}\neq\mathbf{c}_{H'}$ if and only if
$H\neq H'$. Thus the points $\mathbf{c}_{H}$ with $H\in\mathcal{H}(M)$
form a matroid on $\mathcal{H}(M)$, which we denote by $\mathrm{Adj}(M,L)$.
It depends on the realization $L$ of $M$, as the following example
shows.
\begin{example}
Consider the two arrangements of six lines in $\mathbb{P}^{2}$ over
$\mathbb{R}$ shown in Figure \ref{fig:Adj-real}, which, from viewpoint
\ref{enu:HA-viewpoint}, realize the same matroid $M$. Let $L$ and
$L'$ be realization matrices of $M$ given by the left and right
arrangements, respectively. In each arrangement, the three points
marked in red are the points $\mathbf{c}_{H}$ for the mp-flats $H=15$,
$24$, and $36$ of $M$. They are collinear in the left arrangement
but not in the right, so $\left\{ 15,24,36\right\} $ has rank $2$
in $\mathrm{Adj}(M,L)$ and $3$ in $\mathrm{Adj}(M,L')$. Hence $\mathrm{Adj}(M,L)$
and $\mathrm{Adj}(M,L')$ differ. 
\begin{figure}[h]
\begin{spacing}{0.4}
\begin{centering}
\begin{center}
\begin{tikzpicture}[font=\scriptsize]

\begin{scope}[xshift=0cm,yshift=0cm,line cap=round,scale=0.75]
 \path (0,0)--++(90:2) coordinate (A1);
 \path (0,0)--++(210:2) coordinate (A2);
 \path (0,0)--++(-30:2) coordinate (A3);

 \path (0,0)--++(30:1) coordinate (B1);
 \path (0,0)--++(150:1) coordinate (B2);
 \path (0,0)--++(270:1) coordinate (B3);

 \draw [name path=line1][thick] (A1)++(60:1.6)--(A2)--++(240:0.6)node[below=0pt]{$1$};
 \draw [thick] (A2)++(180:0.6)node[left=-1pt]{$3$}--(A3)--++(0:0.6);
 \draw [name path=line5][thick] (A3)++(300:0.6)--(A1)--++(120:0.6)node[above left=-3pt]{$5$};

 \draw [thick,dotted,gray] (A1)++(90:1.4)--(B3)--++(270:1);
 \draw [name path=line2][thick] (A2)++(210:0.6)node[below left=-3pt]{$2$}--(B1)--++(30:0.5);
 \draw [name path=line4][thick] (A3)++(330:0.6)--(B2)--++(150:0.5)node[above left=-2pt]{$4$};
 \draw [name path=line6][thick] (B3)++(260:0.6)node[below left=-3pt]{$6$}--++(80:5);

  \fill [name intersections={of=line6 and line1}] (intersection-1) circle (2.5pt);
 \fill [name intersections={of=line6 and line2}] (intersection-1) circle (2.5pt);
 \fill [name intersections={of=line6 and line4}] (intersection-1) circle (2.5pt);
 \fill [name intersections={of=line6 and line5}] (intersection-1) circle (2.5pt);

 \foreach \x in {(A1),(A2),(A3),(B1),(B2),(B3),(0,0)}{
    \fill [black] \x circle (2.5pt);}
 \foreach \x in {(A1),(B3),(0,0)}{
    \fill [red] \x circle (2.5pt);}

\end{scope}

\begin{scope}[xshift=6.5cm,yshift=0cm,line cap=round,scale=0.75]
 \path (0,0)--++(90:2) coordinate (A1);
 \path (0,0)--++(210:2) coordinate (A2);
 \path (0,0)--++(-30:2) coordinate (A3);

 \path (0,0)--++(30:1) coordinate (B1);
 \path (0,0)--++(150:1) coordinate (B2);
 \path (0,0)--++(270:1) coordinate (B3);

 \draw [name path=line1][thick] (A1)++(60:1.6)--(A2)--++(240:0.6)node[below=0pt]{$1$};
 \draw [name path=line3][thick] (A2)++(180:0.6)node[left=-1pt]{$3$}--(A3)--++(0:0.6);
 \draw [name path=line5][thick] (A3)++(300:0.6)--(A1)--++(120:0.6)node[above left=-3pt]{$5$};

 \draw [thick,dotted,gray] (A1)++(90:1.4)--(B3)--++(270:1);
 \draw [name path=line2][thick] (A2)++(210:0.6)node[below left=-3pt]{$2$}--(B1)--++(30:0.5);
 \draw [name path=line4][thick] (A3)++(330:0.6)--(B2)--++(150:0.5)node[above left=-2pt]{$4$};
 \draw [name path=line6][thick] (A1)++(60:1.52)--++(252:5.2)node[below left=-3pt]{$6$};

  \fill [name intersections={of=line6 and line1}] (intersection-1) circle (2.5pt);
 \fill [name intersections={of=line6 and line2}] (intersection-1) circle (2.5pt);
 \fill [name intersections={of=line6 and line4}] (intersection-1) circle (2.5pt);
 \fill [name intersections={of=line6 and line5}] (intersection-1) circle (2.5pt);
 \fill [name intersections={of=line6 and line3}][red] (intersection-1) circle (2.5pt);

 \foreach \x in {(A1),(A2),(A3),(B1),(B2),(0,0)}{
    \fill [black] \x circle (2.5pt);}
 \foreach \x in {(A1),(0,0)}{
    \fill [red] \x circle (2.5pt);}

\end{scope}

\end{tikzpicture}
\par\end{center}
\par\end{centering}
\end{spacing}
\begin{spacing}{0.3}
\noindent\caption{\label{fig:Adj-real}Two realizations of $M$ giving different adjoint
matroids.}
\end{spacing}
\end{figure}
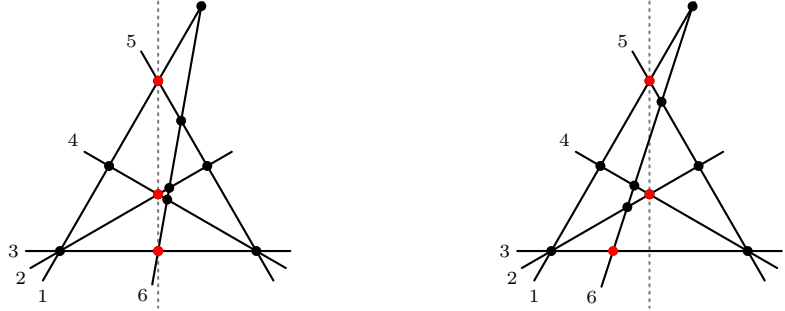
\end{example}

\begin{rem}
The matroid $\mathrm{Adj}(M,L)$ coincides with the type I adjoint
of Bixby and Coullard \cite[Example 2.6]{BC88}, with the matroid
of the derived code of Jurrius and Pellikaan \cite[Definition 4.1]{JP15},
and with the cocircuit matroid described by Kung \cite{Kung}. We
call $\mathrm{Adj}(M,L)$ the \emph{adjoint matroid} of $(M,L)$. 
\end{rem}

Following Cheung \cite{Cheung}, an \emph{adjoint} of $M$ is a simple
matroid $M'$ of the same rank for which there exists an injective
order-preserving map from the inverted flat lattice $\mathcal{L}(M)'$
to $\mathcal{L}(M')$ sending the atoms of $\mathcal{L}(M)'$ onto
those of $\mathcal{L}(M')$. From viewpoint \ref{enu:HA-viewpoint},
one sees that $\mathrm{Adj}(M,L)$ is an adjoint of $M$.

Since the vectors $\mathbf{c}_{H}$ realize $\mathrm{Adj}(M,L)$,
the construction can be iterated. Let $M^{(0)}$ denote the simplification
of $M$, so that $\mathcal{L}(M^{(0)})$ is isomorphic to $\mathcal{L}(M)$,
and let $L^{(0)}$ be a submatrix of $L$ consisting of one column
from each parallel class of $M$. For $i\ge1$, put $M^{(i)}:=\mathrm{Adj}(M^{(i-1)},L^{(i-1)})$,
and let $L^{(i)}$ be a matrix whose columns are representative vectors
$\mathbf{c}_{H}$ for the mp-flats $H$ of $M^{(i-1)}$, so that $L^{(i)}$
realizes $M^{(i)}$ and the sequence $M^{(1)},M^{(2)},\dots$ is determined
by $L$. Kung \cite{Kung} calls the sequence $M^{(0)},M^{(1)},M^{(2)},\dots$
the \emph{Crapo sequence}.

The embedding of $\mathcal{L}(M^{(i)})'$ into $\mathcal{L}(M^{(i+1)})$
sends each pair of flats of $M^{(i)}$ to a modular pair of flats
of $M^{(i+1)}$, and it is surjective if and only if $M^{(i)}$ is
modular. Moreover, $M^{(2i)}$ is an extension of $M^{(2i-2)}$ and
$M^{(2i+1)}$ is an extension of $M^{(2i-1)}$, as observed by Kung
\cite{Kung}; see also \cite[Section 5]{Adjoint}. Hence both sequences
$M^{(0)},M^{(2)},\dots$ and $M^{(1)},M^{(3)},\dots$ \emph{stabilize},
that is, become constant from some term on, if and only if $M^{(t)}$
is modular for some $t$. If $\mathbb{F}$ is finite with $q$ elements,
then each $M^{(i)}$ is a restriction of the projective geometry $\mathrm{PG}(k-1,q)$,
so both sequences stabilize. In this case, the matroid at which $M^{(0)},M^{(2)},\dots$
stabilizes is a simple modular extension of $M^{(0)}$ of the same
rank, and the largest such extension realizable over $\mathbb{F}$
is $\mathrm{PG}(k-1,q)$, cf.~\cite{Kung}.
\begin{example}
Consider the four points in $\mathbb{P}^{2}$ shown in the left panel
of Figure \ref{fig:Adj-ext}. From viewpoint \ref{enu:Span-viewpoint},
they realize $M^{(0)}=U_{3,4}$. The mp-flats of $M^{(0)}$ correspond
to the six lines, each passing through exactly two of the points,
and from viewpoint \ref{enu:HA-viewpoint}, these lines realize $M^{(1)}$,
as shown in the middle panel. Besides the four original points, the
six lines meet at three further points, and from viewpoint \ref{enu:Span-viewpoint},
these seven points realize $M^{(2)}$, as shown in the right panel.

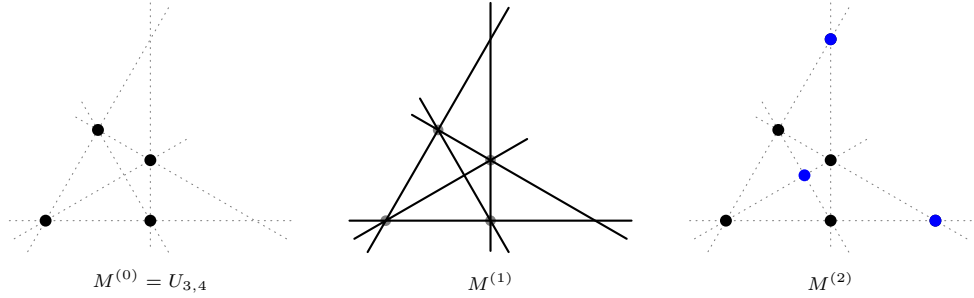
\begin{figure}[h]
\begin{spacing}{0.4}
\begin{centering}
\begin{center}
\begin{tikzpicture}[font=\scriptsize]

\begin{scope}[xshift=0cm,yshift=0cm,line cap=round,scale=0.8]
 \path (0,0)--++(90:2) coordinate (A1);
 \path (0,0)--++(210:2) coordinate (A2);
 \path (0,0)--++(-30:2) coordinate (A3);

 \path (0,0)--++(30:1) coordinate (B1);
 \path (0,0)--++(150:1) coordinate (B2);
 \path (0,0)--++(270:1) coordinate (B3);

 \draw [name path=line1][dotted,gray] (A1)++(60:0.6)--(A2)--++(240:0.6);
 \draw [dotted,gray] (A2)++(180:0.6)--(A3)--++(0:0.6);

 \draw [dotted,gray] (A1)++(90:0.6)--(B3)--++(270:0.5);
 \draw [name path=line2][dotted,gray] (A2)++(210:0.6)--++(30:3.3);
 \draw [name path=line4][dotted,gray] (A3)++(330:0.6)--(B2)--++(150:0.5);

 \draw [dotted,gray] (B3)++(300:0.6)--(B2)--++(120:0.6);

 \foreach \x in {(A2),(B2),(B3),(0,0)}{
    \fill [black] \x circle (2.8pt);}

 \path (270:2) node[]{$M^{(0)}=U_{3,4}$};

\end{scope}

\begin{scope}[xshift=4.5cm,yshift=0cm,line cap=round,scale=0.8]
 \path (0,0)--++(90:2) coordinate (A1);
 \path (0,0)--++(210:2) coordinate (A2);
 \path (0,0)--++(-30:2) coordinate (A3);

 \path (0,0)--++(30:1) coordinate (B1);
 \path (0,0)--++(150:1) coordinate (B2);
 \path (0,0)--++(270:1) coordinate (B3);

 \foreach \x in {(A2),(B2),(B3),(0,0)}{
    \fill [gray] \x circle (2.5pt);}

 \draw [name path=line1][thick] (A1)++(60:0.6)--(A2)--++(240:0.6);
 \draw [thick] (A2)++(180:0.6)--(A3)--++(0:0.6);

 \draw [thick] (A1)++(90:0.6)--(B3)--++(270:0.5);
 \draw [name path=line2][thick] (A2)++(210:0.6)--++(30:3.3);
 \draw [name path=line4][thick] (A3)++(330:0.6)--(B2)--++(150:0.5);

 \draw [thick] (B3)++(300:0.6)--(B2)--++(120:0.6);

 \path (270:2) node[]{$M^{(1)}$};

\end{scope}

\begin{scope}[xshift=9cm,yshift=0cm,line cap=round,scale=0.8]
 \path (0,0)--++(90:2) coordinate (A1);
 \path (0,0)--++(210:2) coordinate (A2);
 \path (0,0)--++(-30:2) coordinate (A3);

 \path (0,0)--++(30:1) coordinate (B1);
 \path (0,0)--++(150:1) coordinate (B2);
 \path (0,0)--++(270:1) coordinate (B3);

 \draw [dotted,gray] (A1)++(60:0.6)--(A2)--++(240:0.6);
 \draw [dotted,gray] (A2)++(180:0.6)--(A3)--++(0:0.6);

 \draw [dotted,gray] (A1)++(90:0.6)--(B3)--++(270:0.5);
 \draw [name path=line2][dotted,gray] (A2)++(210:0.6)--++(30:3.3);
 \draw [dotted,gray] (A3)++(330:0.6)--(B2)--++(150:0.5);

 \draw [name path=line5][dotted,gray] (B3)++(300:0.6)--(B2)--++(120:0.6);

 \foreach \x in {(A2),(B2),(B3),(0,0),(A1),(A3)}{
    \fill [black] \x circle (2.8pt);}
 \foreach \x in {(A1),(A3)}{
    \fill [blue] \x circle (2.8pt);}

 \fill [name intersections={of=line2 and line5}][blue] (intersection-1) circle (2.8pt);

 \path (270:2) node[]{$M^{(2)}$};

\end{scope}

\end{tikzpicture}
\par\end{center}
\par\end{centering}
\end{spacing}
\begin{spacing}{0.3}
\noindent\caption{\label{fig:Adj-ext}$M^{(2)}$ is an extension of $M^{(0)}=U_{3,4}$.}
\end{spacing}
\end{figure}
\end{example}

\subsection{Duality}

Let $\check{L}$ be an $(n-k)\times n$ matrix whose rows span $V^{\perp}$.
This matrix is a realization of the dual matroid $M^{\ast}$. Since
$\dim V^{\perp}=n-k$, the rows $\mathbf{w}_{1},\dots,\mathbf{w}_{n-k}$
of $\check{L}$ are linearly independent, so every circuit vector
$\mathbf{v}_{C}$ can be written uniquely as $\mathbf{v}_{C}=\sum_{i\in\left[n-k\right]}\mathbf{b}_{C}(i)\mathbf{w}_{i}$
with $\mathbf{b}_{C}\in\mathbb{F}^{n-k}$. Each $\mathbf{b}_{C}$
is sent to $\mathbf{v}_{C}$ under the isomorphism $\mathbb{F}^{n-k}\rightarrow V^{\perp}$,
$\mathbf{b}\mapsto\mathbf{b}^{T}\check{L}$, so $\left\{ \mathbf{v}_{C}:C\in\mathcal{C}(M)\right\} $
and $\left\{ \mathbf{b}_{C}:C\in\mathcal{C}(M)\right\} $ form the
same matroid on $\mathcal{C}(M)$. For $j\in\left[n\right]$, we have
$\mathbf{v}_{C}(j)=(\mathbf{b}_{C})^{T}\check{L}_{j}$, which vanishes
if and only if $j\in\left[n\right]\backslash C$. Since $\left[n\right]\backslash C$
is an mp-flat of $M^{\ast}$, the vectors $\mathbf{b}_{C}$ and $\mathbf{c}_{\left[n\right]\backslash C}\in\mathbb{F}^{n-k}$
represent the same point in $\mathbb{P}^{n-k-1}$. Hence the matroid
formed by $\left\{ \mathbf{b}_{C}:C\in\mathcal{C}(M)\right\} $ is
$\mathrm{Adj}(M^{\ast},\check{L})$, and 
\[
\delta_{OW}(M,L)\simeq\mathrm{Adj}(M^{\ast},\check{L})
\]
 under the correspondence $C\mapsto\left[n\right]\backslash C$. This
is \cite[Proposition 6.4]{Adjoint}, stated there in terms of the
type I adjoint. 

\section{\label{sec:Generating-Matroids}Generating Matroids from a Collection
of Sets}

Let $E$ be a finite set. A \emph{seed} is a collection of nonempty
subsets of $E$. From a seed we generate a matroid on $E$ whose dependent-set
collection includes the seed. We recall the dependent-set axioms of
Brylawski \cite{DepAxm}: a collection $\mathfrak{D}\subseteq2^{E}$
is the dependent-set collection of a matroid on $E$ if and only if 
\begin{enumerate}[label=(D\arabic*),start=0]
\item \label{enu:depen0}$\emptyset\notin\mathfrak{D}$. 
\item \label{enu:depen1}If $D\in\mathfrak{D}$ and $D\subseteq D'\subseteq E$,
then $D'\in\mathfrak{D}$.
\item \label{enu:depen2}If $D_{1},D_{2}\in\mathfrak{D}$, then either $D_{1}\cap D_{2}\in\mathfrak{D}$,
or $(D_{1}\cup D_{2})\backslash\{e\}\in\mathfrak{D}$ for all $e\in D_{1}\cap D_{2}$.
\end{enumerate}
Let $\mathfrak{X}\subseteq2^{E}$. The \emph{upward closure} of $\mathfrak{X}$
is 
\[
\mathrm{uc}(\mathfrak{X})=\left\{ Y\subseteq E:X\subseteq Y\text{ for some }X\in\mathfrak{X}\right\} .
\]
 The operator $\mathrm{uc}$ satisfies $\mathfrak{X}\subseteq\mathrm{uc}(\mathfrak{X})$
and is monotone, that is, $\mathrm{uc}(\mathfrak{X})\subseteq\mathrm{uc}(\mathfrak{Y})$
whenever $\mathfrak{X}\subseteq\mathfrak{Y}\subseteq2^{E}$. The \emph{elimination
extension} of $\mathfrak{X}$ is 
\[
\mathrm{ex}(\mathfrak{X})=\mathfrak{X}\cup\left\{ (X_{1}\cup X_{2})\backslash\{e\}:X_{1},X_{2}\in\mathfrak{X},X_{1}\cap X_{2}\notin\mathfrak{X},e\in X_{1}\cap X_{2}\right\} .
\]
 Clearly, $\mathfrak{X}\subseteq\mathrm{ex}(\mathfrak{X})$. These
operators appear in \cite{CDM} in the case where $E$ is the circuit
collection of a matroid.\medskip{}

Define $\Phi=\mathrm{uc}\circ\mathrm{ex}$ and set 
\[
\Phi^{\infty}(\mathfrak{X})=\bigcup\nolimits_{i\ge0}\Phi^{i}(\mathfrak{X}).
\]
 Since $\mathfrak{Y}\subseteq\mathrm{ex}(\mathfrak{Y})\subseteq\mathrm{uc}(\mathrm{ex}(\mathfrak{Y}))=\Phi(\mathfrak{Y})$
for every $\mathfrak{Y}\subseteq2^{E}$, the sequence $\ensuremath{(\Phi^{i}(\mathfrak{X}))_{i\ge0}}$
is increasing, and it stabilizes because $E$ is finite. Hence $\Phi^{\infty}(\mathfrak{X})=\Phi^{t}(\mathfrak{X})$
for some $t\ge0$, and $\Phi^{\infty}(\mathfrak{X})$ is a fixed point
of $\Phi$. If $\emptyset\in\mathfrak{X}$, then $\Phi(\mathfrak{X})=2^{E}$.
By \ref{enu:depen1}, \ref{enu:depen2}, and Proposition \ref{prop:seed}
below, a seed $\mathfrak{X}$ is the dependent-set collection of a
matroid if and only if $\Phi^{\infty}(\mathfrak{X})=\mathfrak{X}$.
\begin{rem}
\label{rem:not-monotone}The operators $\Phi$ and $\Phi^{\infty}$
are not monotone. Let $E=\{1,2,3\}$, $\mathfrak{X}=\left\{ \{1,2\},\{2,3\}\right\} $,
and $\mathfrak{Y}=\mathfrak{X}\cup\left\{ \{2\}\right\} $. Then $\Phi^{\infty}(\mathfrak{X})=\Phi(\mathfrak{X})$
is the dependent-set collection of $U_{1,3}$, and $\Phi^{\infty}(\mathfrak{Y})=\Phi(\mathfrak{Y})$
is that of the matroid on $E$ in which $2$ is a loop and $\{1,3\}$
is independent. Neither includes the other, although $\mathfrak{X}\subseteq\mathfrak{Y}$. 
\end{rem}

Define the \emph{girth} of $\mathfrak{X}$ to be $\mathrm{girth}(\mathfrak{X})=\min\left\{ \left|X\right|:X\in\mathfrak{X}\right\} $,
with $\mathrm{girth}(\mathfrak{X})=\infty$ when $\mathfrak{X}$ is
empty. If $\mathfrak{X}$ is the dependent-set collection of a matroid,
this agrees with the usual girth. In these terms, a seed is precisely
a collection of girth at least $1$. 
\begin{prop}
\label{prop:seed}For a seed $\mathfrak{X}$, the collection $\Phi^{\infty}(\mathfrak{X})$
is the dependent-set collection of a matroid on $E$, and $\mathrm{girth}(\Phi^{\infty}(\mathfrak{X}))=\mathrm{girth}(\mathfrak{X})$. 
\end{prop}

\begin{proof}
Put $\mathfrak{D}=\Phi^{\infty}(\mathfrak{X})$. We first prove that
$\mathrm{girth}(\mathfrak{D})=\mathrm{girth}(\mathfrak{X})$. The
inclusion $\mathfrak{X}\subseteq\mathfrak{D}$ gives $\mathrm{girth}(\mathfrak{D})\le\mathrm{girth}(\mathfrak{X})$,
and for the reverse inequality, it suffices to show that $\mathrm{girth}(\Phi(\mathfrak{Y}))\ge\mathrm{girth}(\mathfrak{Y})$
for every $\mathfrak{Y}\subseteq2^{E}$. Let $X_{1},X_{2}\in\mathfrak{Y}$
with $X_{1}\cap X_{2}\notin\mathfrak{Y}$. Then $X_{1}\neq X_{2}$,
so $X_{1}\cup X_{2}$ properly includes $X_{1}$ or $X_{2}$; hence
$\left|X_{1}\cup X_{2}\right|\ge\mathrm{girth(\mathfrak{Y})}+1$ and
$\left|\left(X_{1}\cup X_{2}\right)\backslash\{e\}\right|\ge\mathrm{girth(\mathfrak{Y})}$
for every $e\in X_{1}\cap X_{2}$. Thus $\mathrm{girth(\mathrm{ex}(\mathfrak{Y}))}\ge\mathrm{girth(\mathfrak{Y})}$,
and the claim follows since every member of $\Phi(\mathfrak{Y})=\mathrm{uc}(\mathrm{ex}(\mathfrak{Y}))$
includes a member of $\mathrm{ex}(\mathfrak{Y})$.

We now verify \ref{enu:depen0}–\ref{enu:depen2} for $\mathfrak{D}$.
Since $\mathfrak{X}$ is a seed, $\mathrm{girth}(\mathfrak{D})=\mathrm{girth}(\mathfrak{X})\ge1$,
so $\emptyset\notin\mathfrak{D}$. This gives \ref{enu:depen0}. Moreover,
$\mathfrak{D}=\Phi(\mathfrak{D})=\mathrm{uc}(\mathrm{ex}(\mathfrak{D}))\supseteq\mathrm{uc}(\mathfrak{D})$,
so $\mathrm{uc}(\mathfrak{D})=\mathfrak{D}$, giving \ref{enu:depen1}.
Let $D_{1},D_{2}\in\mathfrak{D}$ with $D_{1}\cap D_{2}\notin\mathfrak{D}$.
Then $\left(D_{1}\cup D_{2}\right)\backslash\{e\}\in\mathrm{ex}(\mathfrak{D})\subseteq\mathrm{uc}(\mathrm{ex}(\mathfrak{D}))=\Phi(\mathfrak{D})=\mathfrak{D}$
for every $e\in D_{1}\cap D_{2}$, giving \ref{enu:depen2}.
\end{proof}

\begin{defn}
\label{def:GenX}For a seed $\mathfrak{X}$, we denote by $\mathrm{Gen}(\mathfrak{X})$
the matroid on $E$ whose dependent-set collection is $\Phi^{\infty}(\mathfrak{X})$,
and say that $\mathfrak{X}$ \emph{generates} $\mathrm{Gen}(\mathfrak{X})$. 
\end{defn}

If $\mathfrak{X}=\emptyset$, then $\mathrm{ex}(\mathfrak{X})=\mathrm{uc}(\mathfrak{X})=\emptyset$,
so $\Phi^{\infty}(\mathfrak{X})=\emptyset$ and $\mathrm{Gen}(\mathfrak{X})$
is a free matroid.
\begin{prop}
\label{prop:simple-gen}Let $\mathfrak{X}$ be a seed.
\begin{enumerate}
\item \label{enu:girth-1}$\mathrm{Gen}(\mathfrak{X})$ is simple if and
only if $\mathrm{girth}(\mathfrak{X})\ge3$.
\item \label{enu:girth-2}If $\mathfrak{X}$ is nonempty, then $\mathrm{Gen}(\mathfrak{X})$
has rank at least $\mathrm{girth}(\mathfrak{X})-1$.
\end{enumerate}
\end{prop}

\begin{proof}
We have $\mathrm{girth}(\Phi^{\infty}(\mathfrak{X}))=\mathrm{girth}(\mathfrak{X})$
by Proposition \ref{prop:seed}. For (\ref{enu:girth-1}), note that
a matroid is simple if and only if it has no circuit of size at most
$2$, that is, if and only if its dependent-set collection has girth
at least $3$. For (\ref{enu:girth-2}), let $C$ be a member of $\Phi^{\infty}(\mathfrak{X})$
with $\left|C\right|=\mathrm{girth}(\mathfrak{X})$, which exists
as $\mathfrak{X}$ is nonempty. By the minimality of $\left|C\right|$,
every proper subset of $C$ is independent in $\mathrm{Gen}(\mathfrak{X})$,
so $\mathrm{Gen}(\mathfrak{X})$ has rank at least $\left|C\right|-1=\mathrm{girth}(\mathfrak{X})-1$. 
\end{proof}

\section{\label{sec:Syzygy}Gaussian Operations and Syzygy Matroids}

We construct a matroid on $\mathcal{C}(M)$ by producing a seed from
the circuits of $M$ and applying the construction of Section \ref{sec:Generating-Matroids}.
The seed records which collections of circuits are spanned by fewer
circuits than they contain, where spanning is understood in the sense
of a matroid-theoretic analogue of Gaussian elimination developed
below.\medskip{}

A loop $e$ of $M$ forms the circuit $\left\{ e\right\} $, which
meets no other circuit of $M$, so loops play no part in the Gaussian
operations. Put $p=\left|\mathrm{cl}_{M}(\emptyset)\right|$. Then
$\mathcal{C}(M)$ is the disjoint union of $\mathcal{C}(M\backslash\mathrm{cl}_{M}(\emptyset))$
and the collection of the $p$ circuits $\left\{ e\right\} $ with
$e\in\mathrm{cl}_{M}(\emptyset)$. We define 
\[
\mathrm{Syz}(M):=\mathrm{Syz}(M\backslash\mathrm{cl}_{M}(\emptyset))\oplus U_{p,p},
\]
 where $U_{p,p}\simeq\bigoplus^{p}_{i=1}U_{1,1}$ is the free matroid
on the collection of these $p$ circuits, and $\mathrm{Syz}(M\backslash\mathrm{cl}_{M}(\emptyset))$
is given by the construction below, since $M\backslash\mathrm{cl}_{M}(\emptyset)$
is loopless. Accordingly, for the rest of this section, $M$ is assumed
to be loopless. Coloops of $M$ lie in no circuit and do not affect
the construction, so they are allowed. Proposition~\ref{prop:simple}
below, together with the simplicity of $U_{p,p}$, shows that $\mathrm{Syz}(M)$
is also simple when $M$ has loops. 

\subsection{Circuit systems}

A \emph{circuit system} is a collection of circuits of $M$. For a
base $B$ of $M$, the \emph{fundamental circuit system} of $B$ is
$\mathcal{C}_{B}:=\left\{ C(e,B):e\in E(M)\backslash B\right\} $.
The map $\mathcal{B}(M)\rightarrow\left\{ \mathcal{C}_{B}:B\in\mathcal{B}(M)\right\} $
defined by $B\mapsto\mathcal{C}_{B}$ is surjective by definition
but need not be injective. 
\begin{example}
Consider an arrangement of four points $P_{1},\dots,P_{4}$ in $\mathbb{P}^{1}$
such that $P_{1}=P_{2}$ and $P_{2},P_{3},P_{4}$ are distinct (see
Figure \ref{fig:PA}), and let $M$ be the rank-$2$ matroid on $[4]$
given by this arrangement from viewpoint \ref{enu:HA-viewpoint}.
Then $\mathcal{B}(M)=\{13,14,23,24,34\}$ and $\mathcal{C}(M)=\{12,134,234\}$.
Moreover, $\mathcal{C}_{13}=\mathcal{C}_{14}=\left\{ 12,134\right\} $,
$\mathcal{C}_{23}=\mathcal{C}_{24}=\left\{ 12,234\right\} $, and
$\mathcal{C}_{34}=\left\{ 134,234\right\} $, so the map $B\mapsto\mathcal{C}_{B}$
is not injective. 
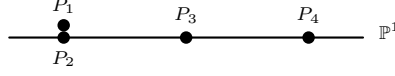
\begin{figure}[h]
\begin{spacing}{0.4}
\begin{centering}
\begin{center}
\begin{tikzpicture}[font=\scriptsize]

\begin{scope}[line cap=round,scale=0.9]

 \path (-1.8,0.18) coordinate (A1) node[above=1pt]{$P_1$};
 \path (-1.8,0) coordinate (A2) node[below=1.5pt]{$P_2$};
 \path (0,0) coordinate (A3) node[above=2pt]{$P_3$};
 \path (1.8,0) coordinate (A4) node[above=2pt]{$P_4$};

 \draw [thick] (-2.6,0)--(2.6,0);
 \path (2.7,0.1) node[right]{$\mathbb{P}^1$};

 \foreach \x in {(A1),(A2),(A3),(A4)}{
    \fill [black] \x circle (2.6pt);}

\end{scope}

\end{tikzpicture}
\par\end{center}
\par\end{centering}
\end{spacing}
\begin{spacing}{0.3}
\noindent\caption{\label{fig:PA}An arrangement of four points in $\mathbb{P}^{1}$.}
\end{spacing}
\end{figure}
\end{example}

\begin{prop}
\label{prop:echelon} A circuit system $\mathcal{S}$ of $M$ is a
fundamental circuit system if and only if $\left|\mathcal{S}\right|=n-k$
and each member of $\mathcal{S}$ contains an element lying in no
other member of $\mathcal{S}$.
\end{prop}

\begin{proof}
If $\mathcal{S}=\mathcal{C}_{B}$ for a base $B$ of $M$, then $\left|\mathcal{S}\right|=n-k$,
and for $e\in E(M)\backslash B$, the circuit $C(e,B)$ is the only
member of $\mathcal{S}$ that contains $e$. Conversely, write $\mathcal{S}=\left\{ C_{1},\dots,C_{n-k}\right\} $
and choose $e_{i}\in C_{i}\backslash\bigcup_{j\neq i}C_{j}$ for $i\in\left[n-k\right]$.
For $B:=E(M)\backslash\{e_{1},\dots,e_{n-k}\}$, we have $C_{i}\backslash\{e_{i}\}\subseteq B$
for each $i$, so $e_{i}\in\mathrm{cl}_{M}(C_{i}\backslash\{e_{i}\})\subseteq\mathrm{cl}_{M}(B)$.
Hence $B$ spans $M$, and since $\left|B\right|=k$, it is a base
of $M$. Then $C_{i}\subseteq B\cup\{e_{i}\}$ gives $C_{i}=C(e_{i},B)$,
and therefore $\mathcal{S}=\mathcal{C}_{B}$. 
\end{proof}

\begin{lem}
\label{lem:fund-circ}Let $B$ be a base of $M$. Let $e_{1},e_{2}\in E(M)\backslash B$
be distinct and put $\tilde{B}=B\cup\{e_{1},e_{2}\}$. Let $C_{1},\dots,C_{m}$
be the circuits included in $\tilde{B}$.
\begin{enumerate}
\item \label{enu:fund1}For each $x\in\tilde{B}$, either no $C_{i}$ contains
$x$, or all but exactly one of them do; in the latter case the remaining
$C_{i}$ is the unique circuit included in $\tilde{B}\backslash\{x\}$. 
\item \label{enu:fund2}The sets $\bigl(\bigcup_{i\in[m]}C_{i}\bigr)\backslash C_{j}$,
$j\in[m]$, form a partition of $\bigcup_{i\in[m]}C_{i}$; in particular,
$\bigcup_{i\in[m]}C_{i}=C_{j}\cup C_{j'}$ for all distinct $j,j'\in[m]$.
\item \label{enu:fund3}Two of the $C_{i}$ are disjoint if and only if
$m=2$. 
\end{enumerate}
\end{lem}

\begin{proof}
Since $\tilde{B}$ includes the base $B$, it has rank $k$ and size
$k+2$. The distinct fundamental circuits $C(e_{1},B)$ and $C(e_{2},B)$
are included in $\tilde{B}$, so $m\ge2$. 

(1) Let $x\in\tilde{B}$. The set $\tilde{B}\backslash\{x\}$ has
size $k+1$ and rank $k$ or $k-1$. If $r(\tilde{B}\backslash\{x\})=k-1$,
then $x$ is a coloop of $M|\tilde{B}$, so no $C_{i}$ contains $x$.
If $r(\tilde{B}\backslash\{x\})=k$, then $\tilde{B}\backslash\{x\}$
is a $(k+1)$-element set including a base, so it includes a unique
circuit $C$. Any $C_{i}$ not containing $x$ is included in $\tilde{B}\backslash\{x\}$,
hence equals $C$.

(2) For $j\in[m]$, the set $P_{j}:=\bigl(\bigcup_{i\in[m]}C_{i}\bigr)\backslash C_{j}$
is nonempty since $m\ge2$. By (\ref{enu:fund1}), each $x\in\bigcup_{i\in[m]}C_{i}$
is contained in all but exactly one of $C_{1},\dots,C_{m}$; hence
$x$ lies in exactly one $P_{j}$. Thus $P_{1},\dots,P_{m}$ form
a partition of $\bigcup_{i\in[m]}C_{i}$. For distinct $j,j'\in[m]$,
we have $\bigl(\bigcup_{i\in[m]}C_{i}\bigr)\backslash(C_{j}\cup C_{j'})=P_{j}\cap P_{j'}=\emptyset$,
so $\bigcup_{i\in[m]}C_{i}\subseteq C_{j}\cup C_{j'}$, and the reverse
inclusion is immediate.

(3) Suppose that $m=2$. By (\ref{enu:fund1}), each $x\in\tilde{B}$
lies in at most one $C_{i}$, so $C_{1}\cap C_{2}=\emptyset$. Conversely,
suppose that two of the $C_{i}$, say $C_{1}$ and $C_{2}$, are disjoint
and $m\ge3$. By (\ref{enu:fund1}), every element of $C_{1}$ lies
in $C_{3}$, so $C_{1}\subsetneq C_{3}$, a contradiction. Hence $m=2$. 
\end{proof}

For a matroid $N$ and a subset $X\subseteq E(N)$, the nullity of
$X$ in $N$ is $\mathrm{nullity}_{N}(X):=\left|X\right|-r_{N}(X)$.
\begin{lem}
\label{lem:one-circ}A subset $D$ of $E(M)$ includes exactly one
circuit of $M$ if and only if $\mathrm{nullity}(D)=1$. 
\end{lem}

\begin{proof}
Let $B$ be a base of $M|D$ and write $C_{M|D}(e,B)$ for the fundamental
circuit taken in $M|D$. Suppose that $D$ includes exactly one circuit.
Then $D$ is dependent, so $D\backslash B\neq\emptyset$. If $e_{1},e_{2}\in D\backslash B$
were distinct, then $C_{M|D}(e_{1},B)$ and $C_{M|D}(e_{2},B)$ would
be distinct circuits of $M|D$ and hence of $M$. Thus $\left|D\right|-r(D)=\left|D\backslash B\right|=1$.
Conversely, if $\mathrm{nullity}(D)=1$, then $D=B\cup\{e\}$ for
some $e$, and $C_{M|D}(e,B)$ is the unique circuit of $M$ included
in $D$.
\end{proof}

In particular, if $C_{1}$ and $C_{2}$ are distinct circuits of $M$,
then $\mathrm{nullity}(C_{1}\cup C_{2})\ge2$, since $C_{1}\cup C_{2}$
includes two distinct circuits.
\begin{lem}
\label{lem:two-circ}Let $C_{1}$ and $C_{2}$ be distinct circuits
of $M$ with $C_{1}\cap C_{2}\neq\emptyset$. The following are equivalent.
\begin{enumerate}
\item \label{enu:two-circ-1}$\mathrm{nullity}(C_{1}\cup C_{2})=2$. 
\item \label{enu:two-circ-2}There is $c\in C_{1}\cap C_{2}$ such that
$(C_{1}\cup C_{2})\backslash\{c\}$ includes exactly one circuit.
\item \label{enu:two-circ-3}For all $c\in C_{1}\cap C_{2}$, the set $(C_{1}\cup C_{2})\backslash\{c\}$
includes exactly one circuit.
\item \label{enu:two-circ-4}There exist a base $B$ of $M$ and $e_{1},e_{2}\in E(M)\backslash B$
such that $C_{1}\cup C_{2}\subseteq B\cup\{e_{1},e_{2}\}$ and $C_{i}=C(e_{i},B)$
for $i=1,2$.
\end{enumerate}
\end{lem}

\begin{proof}
Put $D=C_{1}\cup C_{2}$.\textbf{ }For $c\in C_{1}\cap C_{2}$, we
have $r(D\backslash\{c\})=r(D)$ since $c\in C_{1}$, so Lemma \ref{lem:one-circ}
shows that $D\backslash\{c\}$ includes exactly one circuit if and
only if $\mathrm{nullity}(D)=2$. This condition does not involve
$c$, so (\ref{enu:two-circ-1}), (\ref{enu:two-circ-2}), and (\ref{enu:two-circ-3})
are equivalent. 

(\ref{enu:two-circ-1})$\Rightarrow$(\ref{enu:two-circ-4}). Since
$C_{1}$ and $C_{2}$ are distinct circuits, neither includes the
other; choose $e_{1}\in C_{1}\backslash C_{2}$ and $e_{2}\in C_{2}\backslash C_{1}$.
As $e_{2}\notin C_{1}$, we have $C_{1}\backslash\{e_{1}\}\subseteq D\backslash\{e_{1},e_{2}\}$,
so $e_{1}\in\mathrm{cl}(C_{1}\backslash\{e_{1}\})\subseteq\mathrm{cl}(D\backslash\{e_{1},e_{2}\})$;
symmetrically $e_{2}\in\mathrm{cl}(D\backslash\{e_{1},e_{2}\})$.
Hence $r(D\backslash\{e_{1},e_{2}\})=r(D)=\left|D\right|-2$, so $D\backslash\{e_{1},e_{2}\}$
is independent. Extend it to a base $B$ of $M$; since $e_{1},e_{2}\in\mathrm{cl}(D\backslash\{e_{1},e_{2}\})$,
neither lies in $B$. Now $D\subseteq B\cup\{e_{1},e_{2}\}$ and $C_{1}\subseteq B\cup\{e_{1}\}$,
so $C_{1}=C(e_{1},B)$, and symmetrically $C_{2}=C(e_{2},B)$. 

(\ref{enu:two-circ-4})$\Rightarrow$(\ref{enu:two-circ-2}). Let
$C_{3},\dots,C_{m}$ be the circuits included in $\tilde{B}:=B\cup\{e_{1},e_{2}\}$
other than $C_{1}$ and $C_{2}$. Since $C_{1}\cap C_{2}\neq\emptyset$,
Lemma \ref{lem:fund-circ}(\ref{enu:fund3}) gives $m\ge3$. By Lemma
\ref{lem:fund-circ}(\ref{enu:fund2}), the set $\bigl(\bigcup_{i\in[m]}C_{i}\bigr)\backslash C_{3}=D\backslash C_{3}$
is nonempty; choose $c$ in it. By Lemma \ref{lem:fund-circ}(\ref{enu:fund1}),
the element $c$ is contained in all but $C_{3}$; in particular,
$c\in C_{1}\cap C_{2}$, and $C_{3}$ is the unique circuit included
in $D\backslash\{c\}\subseteq\tilde{B}\backslash\{c\}$. 
\end{proof}

\subsection{Gaussian circuit operations}

We formulate a matroid-theoretic analogue of Gaussian elimination
for circuit systems of $M$. Gaussian elimination for a matrix is
carried out by the following three row operations. 
\begin{enumerate}
\item Swap two rows.
\item Multiply a row by a nonzero scalar.
\item Replace a row by the sum of itself and a scalar multiple of another
row.
\end{enumerate}
A circuit system is a collection rather than an ordered list, so swapping
two rows has no effect on it. Scaling a row by a nonzero scalar changes
the row but not its support, and a circuit system records supports
and nothing else. Hence neither operation alters the circuit system,
and the third operation is the only one carrying combinatorial content. 

Suppose that two nonproportional rows are both nonzero at some coordinate.
Adding a suitable multiple of one to the other eliminates that coordinate,
leaving a nonzero vector whose support is included in the set obtained
from the union of the two supports by removing that coordinate. On
the combinatorial side, the circuit elimination axiom asserts that
removing a common element from the union of two distinct circuits
leaves a set that includes a circuit, and we take the passage to such
a circuit as the third operation on circuit systems.
\begin{defn}
Let $\left(C_{1},C_{2}\right)$ be an ordered pair of distinct circuits
of $M$ and let $e\in C_{1}$. If $e\in C_{2}$, let $C_{2}'$ be
a circuit included in $(C_{1}\cup C_{2})\backslash\{e\}$, which exists
by the circuit elimination axiom; otherwise set $C_{2}'=C_{2}$. The
passage from $\left(C_{1},C_{2}\right)$ to $\left(C_{1},C_{2}'\right)$
is a \emph{Gaussian circuit operation} at $e$. 
\end{defn}

The operation is asymmetric in that the element $e$ is chosen from
$C_{1}$ and only $C_{2}$ is replaced. The two circuits remain distinct,
since $e\in C_{1}\backslash C_{2}'$ in either case. The operation
is not invertible in general: $C_{2}'$ need not be uniquely determined,
and even when it is, there need not exist a Gaussian circuit operation
recovering $C_{2}$ from $\left(C_{1},C_{2}'\right)$. We apply such
an operation to a circuit system $\mathcal{A}$ containing $C_{1}$
and $C_{2}$ by replacing $C_{2}$ with $C_{2}'$, and call $C_{1}$
the \emph{pivot} of the operation.\medskip{}

The following proposition shows that any two fundamental circuit systems
are related by a sequence of Gaussian circuit  operations.
\begin{prop}
\label{prop:transform}Let $A$ and $B$ be bases of $M$. Then a
sequence of Gaussian circuit operations transforms $\mathcal{C}_{B}$
into $\mathcal{C}_{A}$.
\end{prop}

\begin{proof}
We may assume $A\neq B$. Fix $b\in B\backslash A$. By the base exchange
axiom, there exists $a\in A\backslash B$ with $B':=(B\backslash\{b\})\cup\{a\}\in\mathcal{B}(M)$.
Write $E(M)\backslash B=\{e_{1},\ldots,e_{n-k}\}$ and $C_{i}:=C(e_{i},B)$,
so that $\mathcal{C}_{B}=\{C_{1},\ldots,C_{n-k}\}$. Since $a\notin B$,
we have $a\in C_{i}$ if and only if $a=e_{i}$. Each $C_{i}$ falls
into exactly one of the following cases.
\begin{enumerate}
\item \label{enu:transform1}If $b\notin C_{i}$, then $a\notin C_{i}$,
since otherwise $a=e_{i}$ and $C_{i}\subseteq(B\cup\{a\})\backslash\{b\}=B'$,
a contradiction. Then $C_{i}\backslash\{e_{i}\}\subseteq B'$ and
$e_{i}\notin B'$, so $C_{i}=C(e_{i},B')$. Thus $C_{i}=C(e_{i},B)=C(e_{i},B')\in\mathcal{C}_{B}\cap\mathcal{C}_{B'}$.
Let $C_{i}'=C_{i}$.
\item \label{enu:transform2}If $a\in C_{i}$, then $b\in C_{i}$ by the
contrapositive of (\ref{enu:transform1}). Since $C_{i}\backslash\{a\}\subseteq B$
and $C_{i}\backslash\{b\}\subseteq B'$, we have $C_{i}=C(a,B)=C(b,B')\in\mathcal{C}_{B}\cap\mathcal{C}_{B'}$.
Let $C_{i}'=C_{i}$. Here $C_{i}$ is the unique circuit lying in
$\mathcal{C}_{B}\cap\mathcal{C}_{B'}$ that contains both $a$ and
$b$; we denote it by $C^{\star}$.
\item \label{enu:transform3}If $b\in C_{i}$ and $a\notin C_{i}$, then
both $B'=\left(B\backslash\{b\}\right)\cup\{a\}$ and $B''=\left(B\backslash\{b\}\right)\cup\{e_{i}\}$
are bases, and $C(e_{i},B')=C(a,B'')=:C_{i}'$. This circuit can also
be obtained as follows. Applying circuit elimination to $C^{\star}$
and $C_{i}$ with $b\in C^{\star}\cap C_{i}$, we obtain a circuit
$C'$ included in $(C^{\star}\cup C_{i})\backslash\{b\}\subseteq B'\cup\{e_{i}\}$,
so $C'=C_{i}'$. Conversely, $(C^{\star}\cup C'_{i})\backslash\{a\}\subseteq B\cup\{e_{i}\}$,
whose unique circuit is $C_{i}=C(e_{i},B)$. Thus Gaussian circuit
operations take $(C^{\star},C_{i})$ to $(C^{\star},C_{i}')$ and
back.
\end{enumerate}
By (\ref{enu:transform1})–(\ref{enu:transform3}), to each $C_{i}=C(e_{i},B)\in\mathcal{C}_{B}$
there corresponds a circuit $C_{i}'=C(e_{i}',B')\in\mathcal{C}_{B'}$
with $e_{i}'=e_{i}$ for $e_{i}\neq a$, and $e_{i}'=b$ for $e_{i}=a$.
Note that $b\in C_{i}$ if and only if $a\in C_{i}'$. Replacing each
$C_{i}$ by $C_{i}'$ yields the fundamental circuit system $\mathcal{C}_{B'}$.

Write $B\backslash A=\left\{ b_{1},\dots,b_{t}\right\} $ and $A\backslash B=\left\{ a_{1},\dots,a_{t}\right\} $,
where $t=k-\left|A\cap B\right|>0$. Let $B_{0}=B$. Applying base
exchange at each step, we obtain a permutation $\tau\in\mathfrak{S}_{t}$
such that $B_{l}=\left(B_{l-1}\backslash\{b_{l}\}\right)\cup\{a_{\tau(l)}\}\in\mathcal{B}(M)$
for each $l\in[t]$ and $B_{t}=A$. We can transform $\mathcal{C}_{B_{l-1}}$
into $\mathcal{C}_{B_{l}}$ by the Gaussian circuit operations described
above, applied with $B_{l-1}$ and $B_{l}$ in place of $B$ and $B'$.
Iterating this for $l=1,\dots,t$ thus transforms $\mathcal{C}_{B}$
into $\mathcal{C}_{A}$ as follows: 
\[
C_{i}=C(e_{i},B)\rightsquigarrow\begin{cases}
C(e_{i},A) & \text{if }e_{i}\notin A\backslash B,\\
C(\epsilon(e_{i}),A) & \text{if }e_{i}\in A\backslash B.
\end{cases}
\]
 Here $\epsilon:A\backslash B\rightarrow B\backslash A$ is the bijection
defined by $\epsilon(a_{\tau(l)})=b_{l}$.
\end{proof}

\subsection{Gaussian operations on dependent sets}

Since the axiom \ref{enu:depen2} is the dependent-set version of
circuit elimination, Gaussian circuit operations extend naturally
to dependent sets: for an ordered pair $(D_{1},D_{2})$ of dependent
sets and $e\in D_{1}\cap D_{2}$, the operation should replace $D_{2}$
by a dependent set included in $(D_{1}\cup D_{2})\backslash\{e\}$.
Unlike the circuit elimination axiom, \ref{enu:depen2} does not guarantee
that such a set exists, so we first determine when one does.
\begin{lem}
\label{lem:elim-ex}Let $D_{1}$ and $D_{2}$ be dependent sets. The
following are equivalent.
\begin{enumerate}
\item \label{enu:elim-ex-1}$(D_{1}\cup D_{2})\backslash\{e\}$ is independent
for some $e\in D_{1}\cap D_{2}$.
\item \label{enu:elim-ex-2}$D_{1}\cup D_{2}$ includes exactly one circuit.
\end{enumerate}
In this case, the unique circuit is included in $D_{1}\cap D_{2}$.
\end{lem}

\begin{proof}
To prove (\ref{enu:elim-ex-2})$\Rightarrow$(\ref{enu:elim-ex-1}),
let $C$ be the unique circuit included in $D_{1}\cup D_{2}$. Since
both $D_{1}$ and $D_{2}$ are dependent, each includes a circuit,
and by uniqueness that circuit is $C$; hence $C\subseteq D_{1}\cap D_{2}$.
Fix $e\in C$; then $e\in D_{1}\cap D_{2}$. If $(D_{1}\cup D_{2})\backslash\{e\}$
were dependent, it would include a circuit distinct from $C$, a contradiction.
Thus $(D_{1}\cup D_{2})\backslash\{e\}$ is independent. 

Conversely, suppose that $(D_{1}\cup D_{2})\backslash\{e\}$ is independent
for some $e\in D_{1}\cap D_{2}$. Let $B$ be a base of $M$ including
$(D_{1}\cup D_{2})\backslash\{e\}$. Then $D_{1}\cup D_{2}\subseteq B\cup\{e\}$,
and $e\notin B$ since $D_{1}\cup D_{2}$ is dependent. If $C$ is
a circuit included in $D_{1}\cup D_{2}$, then $C\subseteq B\cup\{e\}$,
so $C=C(e,B)$. Therefore $D_{1}\cup D_{2}$ includes exactly one
circuit.
\end{proof}

\begin{cor}
\label{cor:Gaussian}Let $D_{1}$ and $D_{2}$ be dependent sets.
The following are equivalent.
\begin{enumerate}
\item $(D_{1}\cup D_{2})\backslash\{e\}$ is dependent for all $e\in D_{1}\cap D_{2}$.
\item $D_{1}\cup D_{2}$ includes at least two circuits.
\end{enumerate}
\end{cor}

\begin{proof}
Negate both statements in Lemma \ref{lem:elim-ex}, noting that $D_{1}\cup D_{2}$
includes at least one circuit.
\end{proof}

\begin{defn}
Let $\left(D_{1},D_{2}\right)$ be an ordered pair of dependent sets
of $M$ such that $D_{1}\cup D_{2}$ includes at least two circuits,
and let $e\in D_{1}$. If $e\in D_{2}$, let $D_{2}'$ be a dependent
set included in $\left(D_{1}\cup D_{2}\right)\backslash\{e\}$, which
exists by Corollary \ref{cor:Gaussian}; otherwise set $D_{2}'=D_{2}$.
The passage from $\left(D_{1},D_{2}\right)$ to $\left(D_{1},D_{2}'\right)$
is a\emph{ Gaussian operation} at $e$. 
\end{defn}

When $D_{1}$ and $D_{2}$ are circuits, the condition that $D_{1}\cup D_{2}$
includes at least two circuits is equivalent to $D_{1}\neq D_{2}$,
so Gaussian circuit operations are Gaussian operations.
\begin{defn}
For a matroid $N$ and $D\subseteq E(N)$, let $D_{\mathrm{lo}}$
and $D_{\mathrm{co}}$ be the sets of loops and coloops of $N|D$,
respectively. We write $\left\lfloor D\right\rfloor _{N}:=D\backslash(D_{\mathrm{lo}}\cup D_{\mathrm{co}})$.
\end{defn}

Since $M$ is loopless, $\left\lfloor D\right\rfloor $ is the set
of elements of $D$ that are contained in some circuit of $M$ included
in $D$, that is, the union of the circuits of $M$ included in $D$.
Consequently, $D$ includes exactly one circuit if and only if $\left\lfloor D\right\rfloor $
is a circuit.

\subsection{Gaussian span}

As motivation, we first consider the realizable case, where a linear
combination of circuit vectors is reduced to a single circuit vector
by successive eliminations. Let $M$ be realized by a matrix $L$
as in Section \ref{sec:Realizations}, with row space $V$. Let $C_{1},\dots,C_{m}$
and $C$ be distinct circuits of $M$ with circuit vectors $\mathbf{v}_{1},\dots,\mathbf{v}_{m}$
and $\mathbf{v}$, respectively, and suppose that $\mathbf{v}=\sum_{i\in[m]}a_{i}\mathbf{v}_{i}$.
Throughout the following process, $\mathbf{x}_{1},\dots,\mathbf{x}_{m}$
are vectors in $V^{\perp}$, initially $\mathbf{x}_{i}=a_{i}\mathbf{v}_{i}$.
Let $\Lambda_{1}:=\left\{ i\in[m]:\mathbf{x}_{i}\neq\mathbf{0}\right\} $.
Then 
\[
\mathbf{v}=\sum_{i\in\Lambda_{1}}\mathbf{x}_{i}.
\]

\noindent\textbf{Step $j$.} If $\mathrm{supp}(\mathbf{x}_{i})\subseteq C$
for every $i\in\Lambda_{j}$, the process stops. Otherwise choose
an index $\lambda_{j}\in\Lambda_{j}$ with $\mathrm{supp}(\mathbf{x}_{\lambda_{j}})\nsubseteq C$
and an element $c_{j}\in\mathrm{supp}(\mathbf{x}_{\lambda_{j}})\backslash C$.
Then $0=\mathbf{v}(c_{j})=\sum_{i\in\Lambda_{j}}\mathbf{x}_{i}(c_{j})$,
so $1=\sum_{i\in\Lambda_{j}\backslash\{\lambda_{j}\}}\frac{-\mathbf{x}_{i}(c_{j})}{\mathbf{x}_{\lambda_{j}}(c_{j})}$
and 
\begin{align*}
\mathbf{v}=\sum_{i\in\Lambda_{j}}\mathbf{x}_{i} & =\sum_{i\in\Lambda_{j}\backslash\{\lambda_{j}\}}\mathbf{x}_{i}+1\cdot\mathbf{x}_{\lambda_{j}}\\
 & =\sum_{i\in\Lambda_{j}\backslash\{\lambda_{j}\}}\mathbf{x}_{i}+\sum_{i\in\Lambda_{j}\backslash\{\lambda_{j}\}}\frac{-\mathbf{x}_{i}(c_{j})}{\mathbf{x}_{\lambda_{j}}(c_{j})}\mathbf{x}_{\lambda_{j}}\\
 & =\sum_{i\in\Lambda_{j}\backslash\{\lambda_{j}\}}\left(\mathbf{x}_{i}-\frac{\mathbf{x}_{i}(c_{j})}{\mathbf{x}_{\lambda_{j}}(c_{j})}\mathbf{x}_{\lambda_{j}}\right).
\end{align*}
 Set $\mathbf{x}_{i}:=\mathbf{x}_{i}-\frac{\mathbf{x}_{i}(c_{j})}{\mathbf{x}_{\lambda_{j}}(c_{j})}\mathbf{x}_{\lambda_{j}}$
for all $i\in\Lambda_{j}\backslash\{\lambda_{j}\}$, and $\Lambda_{j+1}:=\{i\in\Lambda_{j}\backslash\{\lambda_{j}\}:\mathbf{x}_{i}\neq\mathbf{0}\}$.
Then $\mathbf{x}_{i}(c_{j})=0$ for all $i\in\Lambda_{j+1}$, and
\[
\mathbf{v}=\sum_{i\in\Lambda_{j+1}}\mathbf{x}_{i}.
\]

Since each step removes at least one index from $\Lambda_{j}$ and
$\Lambda_{j}$ remains nonempty because $\mathbf{v}\neq\mathbf{0}$,
the process stops at some step $j\le m$. At termination, $\mathrm{supp}(\mathbf{x}_{i})\subseteq C$
for every $i\in\Lambda_{j}$. Each $\mathrm{supp}(\mathbf{x}_{i})$
is dependent, so $\mathrm{supp}(\mathbf{x}_{i})=C$ by minimality
and $\mathbf{x}_{i}$ is a circuit vector of $C$. \vspace{6pt}

We now abstract this process to arbitrary matroids. Let $\mathcal{A}=\left\{ C_{1},\dots,C_{m}\right\} $
be a circuit system of $M$ and $C$ a circuit. The \emph{Gaussian
process} for $\left(\mathcal{A},C\right)$ is defined as follows.
Throughout the process, $X_{1},\dots,X_{m}$ are dependent sets, initially
$X_{i}=C_{i}$. Let $\Lambda_{1}:=[m]$.\vspace{6pt}

\noindent\textbf{Step $j$.} If $\left|\Lambda_{j}\right|\le1$ or
$C=X_{i}$ for some $i\in\Lambda_{j}$, the process stops. Otherwise
choose an index $\lambda_{j}\in\Lambda_{j}$ and an element $c_{j}\in X_{\lambda_{j}}\backslash C$;
the set $X_{\lambda_{j}}$ is called the \emph{pivot} of step $j$.
For each $i\in\Lambda_{j}\backslash\{\lambda_{j}\}$, apply to $(X_{\lambda_{j}},X_{i})$
the Gaussian operation at $c_{j}$ given by 
\[
X_{i}:=\begin{cases}
\left\lfloor \left(X_{i}\cup X_{\lambda_{j}}\right)\backslash\{c_{j}\}\right\rfloor  & \text{if }c_{j}\in X_{i},\\
X_{i} & \text{otherwise}.
\end{cases}
\]
 Then let $\Lambda_{j+1}$ consist of one index from each group of
indices $i\in\Lambda_{j}\backslash\{\lambda_{j}\}$ with the same
$X_{i}$. For a group with at least two indices, replacing the group
by its retained index is called a \emph{merge}, and that index is
called the \emph{representative} of the group. Since the $X_{i}$
are unions of circuits and those with $i\in\Lambda_{j+1}$ are distinct,
the next step is well defined. \vspace{2pt}

\noindent\textbf{Termination.} If $\mathcal{A}=\emptyset$, the process
stops at the first step. Otherwise, since each step removes at least
one index from $\Lambda_{j}$, it stops at some step $j\le m$.\vspace{2pt}

\noindent\textbf{Success.} The process succeeds if $C=X_{i}$ for
some $i\in\Lambda_{j}$ at termination. \vspace{6pt}

An \emph{execution} of the process is the sequence of choices made
at each step: the index $\lambda_{j}$, the element $c_{j}$, and,
whenever indices are merged, their representatives.
\begin{defn}
A circuit $C$ is a \emph{Gaussian combination} of $\mathcal{A}$
if some execution of the Gaussian process for $\left(\mathcal{A},C\right)$
succeeds.
\end{defn}

\begin{lem}
\label{lem:gc-union}If $C$ is a Gaussian combination of $\mathcal{A}$,
then $C\subseteq\bigcup\mathcal{A}$.
\end{lem}

\begin{proof}
Initially $X_{i}=C_{i}\subseteq\bigcup\mathcal{A}$ for every $i$.
Step $j$ replaces each $X_{i}$ by a subset of $X_{i}\cup X_{\lambda_{j}}$
or leaves it unchanged, and the passage from $\Lambda_{j}$ to $\Lambda_{j+1}$
changes no $X_{i}$. Hence $X_{i}\subseteq\bigcup\mathcal{A}$ for
every $i$ throughout the process. Since the process succeeds, $C=X_{i}$
for some $i\in\Lambda_{j}$ at termination, so $C\subseteq\bigcup\mathcal{A}$.
\end{proof}

\begin{prop}
\label{prop:span}For a base $B$ of $M$, any circuit of $M$ is
a Gaussian combination of the fundamental circuit system $\mathcal{C}_{B}$. 
\end{prop}

\begin{proof}
Let $C$ be a circuit of $M$ and choose $e\in C\backslash B$. Since
$C\backslash\{e\}$ is independent, it extends to a base $A$ with
$A\backslash(C\backslash\{e\})\subseteq B$. Then $e\notin A\cup B$
and $C\subseteq A\cup\{e\}$, so $C=C(e,A)\in\mathcal{C}_{A}$. Index
$\mathcal{C}_{B}$ by $E(M)\backslash B$ and run the Gaussian process
for $\left(\mathcal{C}_{B},C\right)$, following the transformation
of Proposition \ref{prop:transform} from $\mathcal{C}_{B}$ to $\mathcal{C}_{A}$.

In the notation of that proof, the passage described in (\ref{enu:transform3})
is effected by the Gaussian operation taking $C_{i}$ to $\left\lfloor (C^{\star}\cup C_{i})\backslash\{b\}\right\rfloor =C(e_{i},B')$,
so a step of the process with pivot $C^{\star}$ and $c_{j}:=b$ replaces
each remaining $X_{i}$ by the circuit assigned to $C_{i}$ in (\ref{enu:transform1})
or (\ref{enu:transform3}). At step $l$, take $\lambda_{l}$ to be
the index of the pivot of the passage from $\mathcal{C}_{B_{l-1}}$
to $\mathcal{C}_{B_{l}}$ and set $c_{l}:=b_{l}$, which lies in $X_{\lambda_{l}}$;
moreover $b_{l}\notin C$, since $b_{l}\notin A\supseteq C\backslash\{e\}$
and $b_{l}\neq e$. The index of the pivot is $a_{\tau(l)}\in A\backslash B$
and is discarded at the same step, while the index correspondence
fixes every other index. Hence, after step $l$, the remaining indices
are those in $\left(E(M)\backslash B\right)\backslash\{a_{\tau(1)},\dots,a_{\tau(l)}\}$,
each $X_{i}$ equals $C(i,B_{l})$, and these circuits are distinct,
so no merge occurs. 

As $e\notin A$, the index $e$ is never the index of a pivot. For
$l\le t$, the indices $e$ and $a_{\tau(l)}$ are distinct, and both
lie in $\Lambda_{l}$, so $\left|\Lambda_{l}\right|\ge2$. Hence,
if the process stops before step $t+1$, it does so with $C=X_{i}$
for some $i$. Otherwise $e\in\Lambda_{t+1}$ and $X_{e}=C(e,A)=C$,
so the process stops at step $t+1$. In either case it succeeds.
\end{proof}

\begin{defn}
Let $\mathcal{A}$ be a circuit system of $M$. Write $\mathrm{sp}(\mathcal{A})$
for the collection of Gaussian combinations of $\mathcal{A}$, and
define $\mathrm{sp}^{i}(\mathcal{A})$ recursively by $\mathrm{sp}^{0}(\mathcal{A})=\mathcal{A}$
and $\mathrm{sp}^{i}(\mathcal{A})=\mathrm{sp}(\mathrm{sp}^{i-1}(\mathcal{A}))$
for $i\ge1$. The \emph{Gaussian span} of $\mathcal{A}$ is 
\[
\mathrm{sp}^{\infty}(\mathcal{A})=\bigcup\nolimits_{i\ge0}\mathrm{sp}^{i}(\mathcal{A}).
\]
\end{defn}

\begin{prop}
\label{prop:sp-basic}Let $\mathcal{A}$ and $\mathcal{A}'$ be circuit
systems of $M$.
\begin{enumerate}
\item \label{enu:sp-basic1}$\mathcal{A}\subseteq\mathrm{sp}(\mathcal{A})$.
\item \label{enu:sp-basic2}If $\mathcal{A}\subseteq\mathcal{A}'$, then
$\mathrm{sp}(\mathcal{A})\subseteq\mathrm{sp}(\mathcal{A}')$.
\end{enumerate}
\end{prop}

\begin{proof}
For (\ref{enu:sp-basic1}), let $C\in\mathcal{A}$. Then $C=X_{i}$
for some $i\in\Lambda_{1}$, so the Gaussian process for $(\mathcal{A},C)$
stops at the first step and succeeds.

For (\ref{enu:sp-basic2}), let $C$ be a Gaussian combination of
$\mathcal{A}$ and consider a successful execution of the Gaussian
process for $\left(\mathcal{A},C\right)$. Index $\mathcal{A}'$ by
extending the given indexing of $\mathcal{A}$, and run the Gaussian
process for $(\mathcal{A}',C)$ with the same choices of $\lambda_{j}$
and $c_{j}$. Whenever a group being merged contains an index of $\mathcal{A}$,
choose as representative the index of $\mathcal{A}$ that survives
this step in the execution for $(\mathcal{A},C)$.

Then every index of $\mathcal{A}$ that survives a step of the execution
for $\left(\mathcal{A},C\right)$ also survives it here, so each $\lambda_{j}$
still lies in $\Lambda_{j}$ when it is chosen. Hence, at each step
where the execution for $(\mathcal{A},C)$ does not stop, $\Lambda_{j}$
here contains at least two indices, so the process for $(\mathcal{A}',C)$
can stop at such a step only with $C=X_{i}$ for some $i$. The update
of each $X_{i}$ depends only on $X_{i}$ and $X_{\lambda_{j}}$,
so each $X_{i}$ with $C_{i}\in\mathcal{A}$ has the same value at
every step as before. Therefore the process for $(\mathcal{A}',C)$
succeeds at the step where the execution for $(\mathcal{A},C)$ stops,
or earlier.
\end{proof}

By Proposition \ref{prop:sp-basic}(\ref{enu:sp-basic1}), the collections
$\mathrm{sp}^{i}(\mathcal{A})$ form an increasing chain. As $\mathcal{C}(M)$
is finite, this chain stabilizes.
\begin{prop}
\label{prop:closure}Let $\mathcal{A}$ and $\mathcal{A}'$ be circuit
systems of $M$.
\begin{enumerate}
\item \label{enu:closure1}$\mathcal{A}\subseteq\mathrm{sp}^{\infty}(\mathcal{A})$.
\item \label{enu:closure2}If $\mathcal{A}\subseteq\mathcal{A}'$, then
$\mathrm{sp}^{\infty}(\mathcal{A})\subseteq\mathrm{sp}^{\infty}(\mathcal{A}')$.
\item \label{enu:closure3}$\mathrm{sp}^{\infty}(\mathrm{sp}^{\infty}(\mathcal{A}))=\mathrm{sp}^{\infty}(\mathcal{A})$. 
\end{enumerate}
\end{prop}

\begin{proof}
(\ref{enu:closure1}) holds since $\mathrm{sp}^{0}(\mathcal{A})=\mathcal{A}$.
For (\ref{enu:closure2}), Proposition \ref{prop:sp-basic}(\ref{enu:sp-basic2})
and induction on $i$ give $\mathrm{sp}^{i}(\mathcal{A})\subseteq\mathrm{sp}^{i}(\mathcal{A}')$
for every $i$, and the claim follows by taking unions. For (\ref{enu:closure3}),
choose $i$ with $\mathrm{sp}^{\infty}(\mathcal{A})=\mathrm{sp}^{i}(\mathcal{A})$.
Then $\mathrm{sp}^{j}(\mathrm{sp}^{\infty}(\mathcal{A}))=\mathrm{sp}^{j}(\mathrm{sp}^{i}(\mathcal{A}))=\mathrm{sp}^{i+j}(\mathcal{A})=\mathrm{sp}^{\infty}(\mathcal{A})$
for every $j$, and taking the union over $j$ gives (\ref{enu:closure3}).
\end{proof}

Thus $\mathrm{sp}^{\infty}$ is a closure operator on $2^{\mathcal{C}(M)}$.
A closure operator is a matroid closure operator precisely when it
satisfies the Mac\,Lane–Steinitz exchange property, which fails for
$\mathrm{sp}^{\infty}$ by Example \ref{exa:Vamos1}(\ref{enu:not-mcl}).
\begin{defn}
Let $\mathcal{A}$ and $\mathcal{A}'$ be circuit systems of $M$.
We say that $\mathcal{A}'$ \emph{Gaussian spans} $\mathcal{A}$ if
$\mathcal{A}\subseteq\mathrm{sp}^{\infty}(\mathcal{A}')$. The \emph{Gaussian
spanning number} $\mathrm{sn}(\mathcal{A})$ of $\mathcal{A}$ is
the least cardinality of a circuit system that Gaussian spans $\mathcal{A}$. 
\end{defn}

In the definition, $\mathcal{A}'$ is not required to satisfy the
stronger condition $\mathrm{sp}^{\infty}(\mathcal{A})=\mathrm{sp}^{\infty}(\mathcal{A}')$.
When $\mathcal{A}=\mathcal{C}(M)$, however, the two conditions agree,
since $\mathrm{sp}^{\infty}(\mathcal{A}')$ is included in $\mathcal{C}(M)$.
\begin{prop}
\label{prop:sn-basic}Let $\mathcal{A}_{1}$ and $\mathcal{A}_{2}$
be circuit systems of $M$.
\begin{enumerate}
\item \label{enu:sn-basic1}$\mathrm{sn}(\mathcal{A}_{1})\le\left|\mathcal{A}_{1}\right|$. 
\item \label{enu:sn-basic2}If $\mathcal{A}_{1}\subseteq\mathcal{A}_{2}$,
then $\mathrm{sn}(\mathcal{A}_{1})\le\mathrm{sn}(\mathcal{A}_{2})$.
\item \label{enu:sn-basic3}$\mathrm{sn}(\mathcal{A}_{1}\cup\mathcal{A}_{2})\le\mathrm{sn}(\mathcal{A}_{1})+\mathrm{sn}(\mathcal{A}_{2})$. 
\end{enumerate}
\end{prop}

\begin{proof}
(\ref{enu:sn-basic1}) holds since $\mathcal{A}_{1}$ Gaussian spans
itself, by Proposition \ref{prop:closure}(\ref{enu:closure1}). For
$i\in\{1,2\}$, take a circuit system $\mathcal{A}_{i}'$ of least
cardinality that Gaussian spans $\mathcal{A}_{i}$. Then (\ref{enu:sn-basic2})
follows from $\mathcal{A}_{1}\subseteq\mathcal{A}_{2}\subseteq\mathrm{sp}^{\infty}(\mathcal{A}_{2}')$,
which gives $\mathrm{sn}(\mathcal{A}_{1})\le\left|\mathcal{A}_{2}'\right|=\mathrm{sn}(\mathcal{A}_{2})$.
For (\ref{enu:sn-basic3}), Proposition \ref{prop:closure}(\ref{enu:closure2})
gives $\mathcal{A}_{1}\cup\mathcal{A}_{2}\subseteq\mathrm{sp}^{\infty}(\mathcal{A}_{1}'\cup\mathcal{A}_{2}')$,
so $\mathrm{sn}(\mathcal{A}_{1}\cup\mathcal{A}_{2})\le\left|\mathcal{A}_{1}'\cup\mathcal{A}_{2}'\right|\le\left|\mathcal{A}_{1}'\right|+\left|\mathcal{A}_{2}'\right|=\mathrm{sn}(\mathcal{A}_{1})+\mathrm{sn}(\mathcal{A}_{2})$. 
\end{proof}

\begin{lem}
\label{lem:sn-upper}For a base $B$ of $M$, the fundamental circuit
system $\mathcal{C}_{B}$ is a minimal circuit system that Gaussian
spans $\mathcal{C}(M)$. In particular, $\mathrm{sn}(\mathcal{C}(M))\le n-k$.
\end{lem}

\begin{proof}
Proposition \ref{prop:span} gives $\mathcal{C}(M)=\mathrm{sp}^{\infty}(\mathcal{C}_{B})$,
so $\mathrm{sn}(\mathcal{C}(M))\le\left|\mathcal{C}_{B}\right|=n-k$.
A proper subsystem $\mathcal{A}$ of $\mathcal{C}_{B}$ omits $C(e,B)$
for some $e\notin B$. Since $C(e,B)$ is the only member of $\mathcal{C}_{B}$
containing $e$, no member of $\mathcal{A}$ contains $e$, and by
Lemma \ref{lem:gc-union}, neither does any member of $\mathrm{sp}^{\infty}(\mathcal{A})$.
Hence $\ensuremath{C(e,B)\notin\mathrm{sp}^{\infty}(\mathcal{A})}$,
and $\mathcal{C}_{B}$ is minimal. 
\end{proof}

\begin{lem}
\label{lem:sp-two}Let $C_{1}$ and $C_{2}$ be distinct circuits
of $M$ and put $D=C_{1}\cup C_{2}$. Then 
\[
\mathrm{sp}^{\infty}(\{C_{1},C_{2}\})=\begin{cases}
\mathcal{C}(M|D) & \text{if }\mathrm{nullity}(D)=2,\\
\{C_{1},C_{2}\} & \text{otherwise}.
\end{cases}
\]
In the first case, $\mathcal{C}(M|D)=\{C_{1},C_{2}\}$ if $C_{1}\cap C_{2}=\emptyset$,
and $\left|\mathcal{C}(M|D)\right|\ge3$ otherwise.
\end{lem}

\begin{proof}
In the Gaussian process for a circuit system with two members, at
most one Gaussian operation is performed. Hence a member of $\mathrm{sp}(\{C_{1},C_{2}\})$
other than $C_{1}$ and $C_{2}$ has the form $\lfloor D\backslash\{c\}\rfloor$
with $c\in C_{1}\cap C_{2}$, and $D\backslash\{c\}$ then includes
exactly one circuit. By Lemma \ref{lem:two-circ}, such a member exists
only if $C_{1}\cap C_{2}\neq\emptyset$ and $\mathrm{nullity}(D)=2$.

Suppose first that $C_{1}\cap C_{2}=\emptyset$ or $\mathrm{nullity}(D)\neq2$.
Then $\mathrm{sp}(\{C_{1},C_{2}\})=\{C_{1},C_{2}\}$, and hence $\mathrm{sp}^{\infty}(\{C_{1},C_{2}\})=\{C_{1},C_{2}\}$.
If in addition $\mathrm{nullity}(D)=2$, then $C_{1}\cap C_{2}=\emptyset$
and $r(C_{1})+r(C_{2})=\left|D\right|-2=r(D)$, so $C_{1}$ and $C_{2}$
are separators of $M|D$ and $\mathcal{C}(M|D)=\{C_{1},C_{2}\}$,
in agreement with the statement.

Suppose next that $C_{1}\cap C_{2}\neq\emptyset$ and $\mathrm{nullity}(D)=2$.
By Lemma \ref{lem:two-circ}, there exist a base $B$ of $M$ and
$e_{1},e_{2}\in E(M)\backslash B$ with $D\subseteq\tilde{B}:=B\cup\{e_{1},e_{2}\}$
and $C_{i}=C(e_{i},B)$ for $i=1,2$. Let $C_{3},\dots,C_{m}$ be
the circuits included in $\tilde{B}$ other than $C_{1}$ and $C_{2}$.
Then $\mathcal{C}(M|D)=\{C_{1},\dots,C_{m}\}$ by Lemma \ref{lem:fund-circ}(\ref{enu:fund2}),
and $m\ge3$ by Lemma \ref{lem:fund-circ}(\ref{enu:fund3}). For
$j\ge3$, Lemma \ref{lem:fund-circ}(\ref{enu:fund2}) gives an element
$c\in D\backslash C_{j}$, and Lemma \ref{lem:fund-circ}(\ref{enu:fund1})
shows that $c\in C_{1}\cap C_{2}$ and $C_{j}$ is the unique circuit
included in $D\backslash\{c\}$; hence $C_{j}=\lfloor D\backslash\{c\}\rfloor\in\mathrm{sp}(\{C_{1},C_{2}\})$.
Thus $\mathrm{sp}(\{C_{1},C_{2}\})=\mathcal{C}(M|D)$, and since every
Gaussian combination of $\mathcal{C}(M|D)$ is included in $D$ by
Lemma \ref{lem:gc-union}, we get $\mathrm{sp}^{\infty}(\{C_{1},C_{2}\})=\mathcal{C}(M|D)$.
\end{proof}

\begin{lem}
\label{lem:sn-12}Let $M$ be connected, and let $i\in\{1,2\}$. Then
$\mathrm{sn}(\mathcal{C}(M))=i$ if and only if $n-k=i$. 
\end{lem}

\begin{proof}
Suppose first that $\mathrm{sn}(\mathcal{C}(M))=1$, so that $\mathcal{C}(M)=\mathrm{sp}^{\infty}(\{C\})$
for some circuit $C$ of $M$. The Gaussian process for a circuit
system with one member stops at the first step, so $\mathrm{sp}(\{C\})=\{C\}$
and hence $\mathcal{C}(M)=\{C\}$. Every element outside $C$ is then
a coloop, so $C=E(M)$ by connectedness and $n-k=1$. 

Suppose next that $\mathrm{sn}(\mathcal{C}(M))=2$, so that $\mathcal{C}(M)=\mathrm{sp}^{\infty}(\{C_{1},C_{2}\})$
for distinct circuits $C_{1}$ and $C_{2}$ of $M$. Then every circuit
of $M$ is included in $C_{1}\cup C_{2}$ by Lemma \ref{lem:gc-union},
so $C_{1}\cup C_{2}=E(M)$ by connectedness. If $\mathrm{nullity}(E(M))\neq2$,
then $\mathcal{C}(M)=\mathrm{sp}^{\infty}(\{C_{1},C_{2}\})=\{C_{1},C_{2}\}$
by Lemma \ref{lem:sp-two}, so $M$ has exactly two circuits and $C_{1}\cap C_{2}=\emptyset$
by circuit elimination, contradicting the connectedness of $M$. Hence
$n-k=2$.

Conversely, $\mathrm{sn}(\mathcal{C}(M))\le n-k$ by Lemma \ref{lem:sn-upper},
and $\mathrm{sn}(\mathcal{C}(M))\neq0$ as $M$ has a circuit. If
$n-k=1$, then $\mathrm{sn}(\mathcal{C}(M))=1$. If $n-k=2$, then
$\mathrm{sn}(\mathcal{C}(M))\neq1$ by the case $i=1$, so $\mathrm{sn}(\mathcal{C}(M))=2$.
\end{proof}

\begin{cor}
\label{cor:sn-3}Let $M$ be connected. Then $\mathrm{sn}(\mathcal{C}(M))\ge3$
if and only if $n-k\ge3$.
\end{cor}

\begin{proof}
Immediate from Lemma \ref{lem:sn-12}, since $\mathrm{sn}(\mathcal{C}(M))=0$
if and only if $n=k$.
\end{proof}

\begin{example}[Vámos matroid]
\label{exa:Vamos1} 
\begin{enumerate}
\item \label{enu:Vamos-DEF}The Vámos matroid $V_{8}$ is the rank-$4$
matroid on an eight-element set partitioned into four pairs, in which
five of the six unions of two pairs are circuits and the remaining
$\binom{8}{4}-5=65$ four-element subsets are bases. Label the pairs
$12$, $34$, $56$, and $78$ so that these circuits are $C_{0}=1234$,
$C_{1}=1256$, $C_{2}=1278$, $C_{3}=3456$, and $C_{4}=3478$. No
two of $C_{0},\ldots,C_{4}$ share more than two elements, so a five-element
subset includes at most one of them. Hence a five-element subset has
rank $4$ and nullity $1$, and it is a circuit precisely when it
is not of the form $C_{i}\cup\left\{ j\right\} $; there are $\binom{8}{5}-5\cdot4=36$
such subsets. There are no circuits of size at most $3$. Altogether,
$V_{8}$ has $41$ circuits.
\item \label{enu:Vamos-sn-3}Since $V_{8}$ is connected and $n-k=4\ge3$,
we have $\mathrm{sn}(\mathcal{C}(V_{8}))\ge3$ by Corollary \ref{cor:sn-3}.
To show the reverse inequality, let $\mathcal{A}=\{C_{0},C_{1},C_{4}\}$
and let $B=5678$, which is a base of $V_{8}$. By Propositions \ref{prop:span} and \ref{prop:sp-basic},
it suffices to show that $\mathcal{C}_{B}\subseteq\mathrm{sp}(\mathcal{A})$.
We run the Gaussian process for $\left(\mathcal{A},C(1,B)\right)$,
where $C(1,B)=15678$. Each set to which $\lfloor\cdot\rfloor$ is
applied below has five elements, so by (\ref{enu:Vamos-DEF}) and
Lemma \ref{lem:one-circ} it includes exactly one circuit, which is
the value of $\lfloor\cdot\rfloor$. At the first step, take the pivot
$C_{0}$ and $c_{1}=3\in C_{0}\backslash C(1,B)$. Since $3\in C_{4}$
and $3\notin C_{1}$, this replaces $C_{4}$ by $\left\lfloor (C_{0}\cup C_{4})\backslash\{3\}\right\rfloor =C_{2}$
and leaves $C_{1}$ unchanged. At the second step, take the pivot
$C_{1}$ and $c_{2}=2\in C_{1}\backslash C(1,B)$. Since $2\in C_{2}$,
this replaces $C_{2}$ by $\left\lfloor (C_{1}\cup C_{2})\backslash\{2\}\right\rfloor =C(1,B)$,
and the process succeeds. Hence $C(1,B)\in\mathrm{sp}(\mathcal{A})$.
Taking $c_{2}=1$ at the second step instead gives $C(2,B)=\left\lfloor (C_{1}\cup C_{2})\backslash\{1\}\right\rfloor $.
Similarly, the pivot $C_{1}$ with $c_{1}=1$ replaces $C_{0}$ by
$\left\lfloor (C_{0}\cup C_{1})\backslash\{1\}\right\rfloor =C_{3}$
and leaves $C_{4}$ unchanged since $1\notin C_{4}$, and a second
step with the pivot $C_{4}$ and $c_{2}=4$ or $c_{2}=3$ gives $C(3,B)=\left\lfloor (C_{3}\cup C_{4})\backslash\{4\}\right\rfloor $
or $C(4,B)=\left\lfloor (C_{3}\cup C_{4})\backslash\{3\}\right\rfloor $,
respectively. The Gaussian processes for $(\mathcal{A},C_{2})$ and
$(\mathcal{A},C_{3})$ succeed after the first steps above, so $C_{2},C_{3}\in\mathrm{sp}(\mathcal{A})$.
Thus $\left\{ C_{0},\dots,C_{4}\right\} \cup\mathcal{C}_{B}\subseteq\mathrm{sp}(\mathcal{A})$,
so $\mathcal{C}(V_{8})=\mathrm{sp}^{2}(\mathcal{A})=\mathrm{sp}^{\infty}(\mathcal{A})$
and $\mathrm{sn}(\mathcal{C}(V_{8}))\le\left|\mathcal{A}\right|=3$.
Therefore $\mathrm{sn}(\mathcal{C}(V_{8}))=3$.
\item \label{enu:not-mcl}The operator $\mathrm{sp}^{\infty}$ is not a
matroid closure operator. Indeed, $\mathcal{C}_{B}$ is a minimal
circuit system that Gaussian spans $\mathcal{C}(V_{8})$ by Lemma
\ref{lem:sn-upper}, and so is $\mathcal{A}$, since a proper subsystem
of $\mathcal{A}$ has at most two members and $\mathrm{sn}(\mathcal{C}(V_{8}))=3$
by (\ref{enu:Vamos-sn-3}). However, $\left|\mathcal{A}\right|=3$
and $\left|\mathcal{C}_{B}\right|=4$, whereas for a matroid closure
operator the minimal sets whose closure is the whole ground set are
the bases, which are equicardinal.
\end{enumerate}
\end{example}

\subsection{Syzygy matroids}

We now construct a matroid on $\mathcal{C}(M)$ from the Gaussian
spanning number. A first candidate for its spanning-set collection
is the collection $\mathfrak{T}=\left\{ \mathcal{A}\subseteq\mathcal{C}(M):\mathrm{sp}^{\infty}(\mathcal{A})=\mathcal{C}(M)\right\} $
of circuit systems that Gaussian span $\mathcal{C}(M)$. Recall the
spanning-set axioms of Brylawski \cite{DepAxm}: a collection $\mathfrak{S}\subseteq2^{E}$
is the spanning-set collection of a matroid on $E$ if and only if 
\begin{enumerate}[label=(S\arabic*),start=0]
\item \label{enu:span0}$E\in\mathfrak{S}$. 
\item \label{enu:span1}If $S\in\mathfrak{S}$ and $S\subseteq S'\subseteq E$,
then $S'\in\mathfrak{S}$.
\item \label{enu:span2}For each $E'\subseteq E$, the minimal members of
$\mathfrak{S}$ including $E'$ are equicardinal.
\end{enumerate}
On the ground set $\mathcal{C}(M)$, the collection $\mathfrak{T}$
satisfies \ref{enu:span0} and \ref{enu:span1} by Proposition \ref{prop:closure}(\ref{enu:closure1})
and (\ref{enu:closure2}), but not \ref{enu:span2} in general, as
Example~\ref{exa:Vamos1}(\ref{enu:not-mcl}) shows. Since Section~\ref{sec:Generating-Matroids}
provides a procedure for adjoining dependent sets but none for adjoining
spanning sets, we turn to dependent sets.\medskip{}

In a matroid, the rank of a subset $X$ is the least size of a subset
whose closure includes $X$, so $X$ is dependent precisely when it
is spanned by fewer elements than it contains. The Gaussian spanning
number is defined in the same way, with Gaussian spans in place of
closures. Accordingly, we take as the seed the collection of circuit
systems that are Gaussian spanned by fewer circuits than they contain,
that is, those $\mathcal{A}$ with $\left|\mathcal{A}\right|>\mathrm{sn}(\mathcal{A})$.
\begin{defn}
Let $\mathfrak{X}=\left\{ \mathcal{A}\subseteq\mathcal{C}(M):\left|\mathcal{A}\right|>\mathrm{sn}(\mathcal{A})\right\} $.
Since $\mathrm{sn}(\emptyset)=0=\left|\emptyset\right|$, we have
$\emptyset\notin\mathfrak{X}$, so $\mathfrak{X}$ is a seed. The
\emph{syzygy matroid} of $M$ is $\mathrm{Syz}(M):=\mathrm{Gen}(\mathfrak{X})$.
\end{defn}

\begin{prop}
\label{prop:syz-sn}If $\mathcal{A}$ is a circuit system of $M$,
then $r_{\mathrm{Syz}(M)}(\mathcal{A})\le\mathrm{sn}(\mathcal{A})$.
\end{prop}

\begin{proof}
Let $\mathcal{A}'$ be an independent set of $\mathrm{Syz}(M)|\mathcal{A}$.
Then $\mathcal{A}'\notin\Phi^{\infty}(\mathfrak{X})$ and $\mathfrak{X}\subseteq\Phi^{\infty}(\mathfrak{X})$,
so $\mathcal{A}'\notin\mathfrak{X}$ and $\left|\mathcal{A}'\right|\le\mathrm{sn}(\mathcal{A}')\le\mathrm{sn}(\mathcal{A})$
by Proposition~\ref{prop:sn-basic}(\ref{enu:sn-basic2}). Taking
$\mathcal{A}'$ to be a base of $\mathrm{Syz}(M)|\mathcal{A}$, we
have $r_{\mathrm{Syz}(M)}(\mathcal{A})=\left|\mathcal{A}'\right|\le\mathrm{sn}(\mathcal{A})$.
\end{proof}

\begin{cor}
\label{cor:rank-sn}$r(\mathrm{Syz}(M))\le\mathrm{sn}(\mathcal{C}(M))\le n-k$.
\end{cor}

\begin{proof}
Immediate from Proposition \ref{prop:syz-sn} and Lemma \ref{lem:sn-upper}. 
\end{proof}

\begin{cor}
\label{cor:sp-cl}If $\mathcal{A}$ is an independent set of $\mathrm{Syz}(M)$,
then $\mathrm{sp}^{\infty}(\mathcal{A})\subseteq\mathrm{cl}_{\mathrm{Syz}(M)}(\mathcal{A})$. 
\end{cor}

\begin{proof}
Let $C\in\mathrm{sp}^{\infty}(\mathcal{A})$. We show that $\ensuremath{r_{\mathrm{Syz}(M)}(\mathcal{A}\cup\{C\})\leq r_{\mathrm{Syz}(M)}(\mathcal{A})}$.
Let $\mathcal{A}'$ be a circuit system of least cardinality with
$\mathcal{A}\subseteq\mathrm{sp}^{\infty}(\mathcal{A}')$, so that
$\left|\mathcal{A}'\right|=\mathrm{sn}(\mathcal{A})$. Proposition~\ref{prop:closure}(\ref{enu:closure2})
and (\ref{enu:closure3}) give $\mathrm{sp}^{\infty}(\mathcal{A})\subseteq\mathrm{sp}^{\infty}(\mathcal{A}')$,
so $\mathcal{A}\cup\{C\}\subseteq\mathrm{sp}^{\infty}(\mathcal{A}')$
and $\mathrm{sn}(\mathcal{A}\cup\{C\})\le\left|\mathcal{A}'\right|=\mathrm{sn}(\mathcal{A})$.
Since $\mathcal{A}$ is independent in $\mathrm{Syz}(M)$, we have
$\left|\mathcal{A}\right|=r_{\mathrm{Syz}(M)}(\mathcal{A})$, and
Propositions \ref{prop:syz-sn} and \ref{prop:sn-basic}(\ref{enu:sn-basic1})
give 
\[
r_{\mathrm{Syz}(M)}(\mathcal{A}\cup\{C\})\le\mathrm{sn}(\mathcal{A}\cup\{C\})\le\mathrm{sn}(\mathcal{A})\le\left|\mathcal{A}\right|=r_{\mathrm{Syz}(M)}(\mathcal{A}).
\]
 Thus $C\in\mathrm{cl}_{\mathrm{Syz}(M)}(\mathcal{A})$, and the inclusion
is proved.
\end{proof}

\begin{prop}
\label{prop:sn-nullity}If $\mathcal{A}$ is a circuit system of $M$,
then $\mathrm{sn}(\mathcal{A})\le\mathrm{nullity}(\bigcup\mathcal{A})$.
\end{prop}

\begin{proof}
Put $X=\bigcup\mathcal{A}$ and $N=M|X$. For a base $B$ of $N$,
let $\mathcal{C}'_{B}$ be its fundamental circuit system in $N$.
Since $\left|\mathcal{C}'_{B}\right|=\left|X\right|-r(X)=\mathrm{nullity}(\bigcup\mathcal{A})$,
it suffices to show that $\mathcal{C}'_{B}$ Gaussian spans $\mathcal{A}$.
For $D\subseteq X$, we have $N|D=M|D$, so $\left\lfloor D\right\rfloor _{N}=\left\lfloor D\right\rfloor _{M}$.
Hence, for a circuit system $\mathcal{A}'$ and a circuit $C$, both
of $N$, the Gaussian processes for $\left(\mathcal{A}',C\right)$
in $N$ and in $M$ coincide, merges included, so a Gaussian combination
in $N$ of a circuit system of $N$ is one in $M$. Every member of
$\mathcal{A}$ is a circuit of $N$, so each is a Gaussian combination
of $\mathcal{C}'_{B}$ in $N$ by Proposition \ref{prop:span} applied
to $N$, hence in $M$. Therefore $\mathcal{C}'_{B}$ Gaussian spans
$\mathcal{A}$.
\end{proof}

\begin{cor}
\label{cor:syz-sn-nullity}If $\mathcal{A}$ is a circuit system of
$M$, then 
\[
r_{\mathrm{Syz}(M)}(\mathcal{A})\le\mathrm{sn}(\mathcal{A})\le\mathrm{nullity}\bigl(\bigcup\mathcal{A}\bigr).
\]
\end{cor}

\begin{proof}
Combine Propositions \ref{prop:syz-sn} and \ref{prop:sn-nullity}.
\end{proof}

The combinatorial derived matroid of \cite{CDM} is generated by the
seed $\{\mathcal{A}\subseteq\mathcal{C}(M):\left|\mathcal{A}\right|>\mathrm{nullity}(\bigcup\mathcal{A})\}$,
which is included in our seed $\mathfrak{X}$ by Proposition~\ref{prop:sn-nullity}.
Since $\Phi^{\infty}$ is not monotone (Remark \ref{rem:not-monotone}),
this inclusion implies no comparison between the generated matroids. 
\begin{prop}
\label{prop:simple}$\mathrm{Syz}(M)$ is simple.
\end{prop}

\begin{proof}
As in the proof of Lemma \ref{lem:sn-12}, we have $\mathrm{sp}^{\infty}(\{C\})=\{C\}$
for $C\in\mathcal{C}(M)$, and $\mathrm{sp}^{\infty}(\emptyset)=\emptyset$.
Hence, if $\mathrm{sn}(\mathcal{A})\leq1$, then $\mathcal{A}$ is
included in a circuit system of cardinality $\mathrm{sn}(\mathcal{A})$,
so $|\mathcal{A}|\leq\mathrm{sn}(\mathcal{A})$. For $\mathcal{A}\in\mathfrak{X}$,
we have $|\mathcal{A}|>\mathrm{sn}(\mathcal{A})$, so $\mathrm{sn}(\mathcal{A})\geq2$
by the above and $|\mathcal{A}|\geq3$. Thus $\mathrm{girth}(\mathfrak{X})\geq3$,
and $\mathrm{Syz}(M)$ is simple by Proposition~\ref{prop:simple-gen}(\ref{enu:girth-1}).
\end{proof}

\begin{prop}
\label{prop:sn-split}Let $M=M_{1}\oplus M_{2}$. For a circuit system
$\mathcal{A}$ of $M$, put $\mathcal{A}_{l}=\mathcal{A}\cap\mathcal{C}(M_{l})$
for $l=1,2$. Then 
\[
\mathrm{sp}^{\infty}(\mathcal{A})=\mathrm{sp}^{\infty}(\mathcal{A}_{1})\sqcup\mathrm{sp}^{\infty}(\mathcal{A}_{2})\quad\text{and}\quad\mathrm{sn}(\mathcal{A})=\mathrm{sn}(\mathcal{A}_{1})+\mathrm{sn}(\mathcal{A}_{2}).
\]
\end{prop}

\begin{proof}
Every circuit of $M$ is included in $E(M_{1})$ or in $E(M_{2})$.
In any execution of the Gaussian process for $(\mathcal{A},C)$, each
$X_{i}$ stays within the $E(M_{l})$ that includes $C_{i}$: initially
$X_{i}=C_{i}$, and a step changes $X_{i}$ only when $c_{j}\in X_{i}\cap X_{\lambda_{j}}$,
in which case $X_{i}$ and $X_{\lambda_{j}}$ are included in the
same $E(M_{l})$, and so is the new $X_{i}$, being included in their
union. In particular, no merge involves indices from different summands.

Let $C\in\mathcal{C}(M_{l})$ be a Gaussian combination of $\mathcal{A}$
and consider a successful execution for $\left(\mathcal{A},C\right)$.
The choices made at the steps with $C_{\lambda_{j}}\in\mathcal{A}_{l}$
form an execution for $\left(\mathcal{A}_{l},C\right)$, along which
each $X_{i}$ with $C_{i}\in\mathcal{A}_{l}$ takes the same values
as before. This execution therefore succeeds, so $\ensuremath{C\in\mathrm{sp}(\mathcal{A}_{l})}$.
Together with Proposition~\ref{prop:sp-basic}(\ref{enu:sp-basic2}),
this gives $\mathrm{sp}(\mathcal{A})=\mathrm{sp}(\mathcal{A}_{1})\sqcup\mathrm{sp}(\mathcal{A}_{2})$,
and the first assertion follows by iteration.

Let $\mathcal{A}'$ be a circuit system of least cardinality that
Gaussian spans $\mathcal{A}$. By the first assertion, $\mathcal{A}'\cap\mathcal{C}(M_{l})$
Gaussian spans $\mathcal{A}_{l}$, and $\mathcal{A}'$ is the disjoint
union of these two systems, so $\mathrm{sn}(\mathcal{A}_{1})+\mathrm{sn}(\mathcal{A}_{2})\le\left|\mathcal{A}'\right|=\mathrm{sn}(\mathcal{A})$;
the reverse inequality follows from Proposition~\ref{prop:sn-basic}(\ref{enu:sn-basic3}).
\end{proof}

\begin{cor}
\label{cor:Syz-split}Let $M=M_{1}\oplus M_{2}$. Then $\mathrm{Syz}(M)=\mathrm{Syz}(M_{1})\oplus\mathrm{Syz}(M_{2})$. 
\end{cor}

\begin{proof}
Write $\mathfrak{X}_{l}$ for the seed generating $\mathrm{Syz}(M_{l})$.
Identify $\mathcal{C}(M)$ with $\mathcal{C}(M_{1})\sqcup\mathcal{C}(M_{2})$,
and call $\mathcal{A}_{l}=\mathcal{A}\cap\mathcal{C}(M_{l})$ the
$l$-th part of a circuit system $\mathcal{A}$ of $M$. A circuit
system of $M$ is dependent in $\mathrm{Syz}(M_{1})\oplus\mathrm{Syz}(M_{2})$
if and only if one of its parts is dependent in the corresponding
summand, so it is enough to show by induction on $i$ that $\mathcal{A}\in\Phi^{i}(\mathfrak{X})$
if and only if $\mathcal{A}_{l}\in\Phi^{i}(\mathfrak{X}_{l})$ for
some $l$.

For $i=0$, the first assertion of Proposition \ref{prop:sn-split}
and the proof of Proposition \ref{prop:sn-nullity} show that $\mathrm{sn}(\mathcal{A}_{l})$
is the same whether computed in $M$ or in $M_{l}$, and the second
assertion of Proposition \ref{prop:sn-split} gives $\left|\mathcal{A}\right|-\mathrm{sn}(\mathcal{A})=\left(\left|\mathcal{A}_{1}\right|-\mathrm{sn}(\mathcal{A}_{1})\right)+\left(\left|\mathcal{A}_{2}\right|-\mathrm{sn}(\mathcal{A}_{2})\right)$.
Each term on the right is nonnegative by Proposition~\ref{prop:sn-basic}(\ref{enu:sn-basic1}),
so $\mathcal{A}\in\mathfrak{X}$ if and only if $\mathcal{A}_{l}\in\mathfrak{X}_{l}$
for some $l$. 

Suppose the claim holds for $i$ and let $\mathcal{A}\in\Phi^{i+1}(\mathfrak{X})$.
Then $\mathcal{A}$ includes a member $\mathcal{T}$ of $\mathrm{ex}(\Phi^{i}(\mathfrak{X}))$.
We show that $\mathcal{T}_{l}\in\Phi^{i+1}(\mathfrak{X}_{l})$ for
some $l$; since $\mathcal{A}_{l}$ includes $\mathcal{T}_{l}$ and
$\Phi^{i+1}(\mathfrak{X}_{l})$ is upward closed, this gives $\mathcal{A}_{l}\in\Phi^{i+1}(\mathfrak{X}_{l})$.
If $\mathcal{T}\in\Phi^{i}(\mathfrak{X})$, then $\mathcal{T}_{l}\in\Phi^{i}(\mathfrak{X}_{l})\subseteq\Phi^{i+1}(\mathfrak{X}_{l})$
for some $l$ by the induction hypothesis. Otherwise $\mathcal{T}=\left(\mathcal{S}\cup\mathcal{S}'\right)\backslash\{C\}$
with $\mathcal{S},\mathcal{S}'\in\Phi^{i}(\mathfrak{X})$, $\mathcal{S}\cap\mathcal{S}'\notin\Phi^{i}(\mathfrak{X})$,
and $C\in\mathcal{S}\cap\mathcal{S}'$. The circuit $C$ lies in $\mathcal{C}(M_{1})$
or in $\mathcal{C}(M_{2})$, say the former. If $\mathcal{S}_{2}$
or $\mathcal{S}_{2}'$ lies in $\Phi^{i}(\mathfrak{X}_{2})$, then
so does $\mathcal{T}_{2}=\mathcal{S}_{2}\cup\mathcal{S}_{2}'$, since
$\Phi^{i}(\mathfrak{X}_{2})$ is upward closed (by definition for
$i\ge1$, and by Proposition~\ref{prop:sn-basic}(\ref{enu:sn-basic1})
and (\ref{enu:sn-basic3}) for $i=0$). Otherwise $\mathcal{S}_{1},\mathcal{S}_{1}'\in\Phi^{i}(\mathfrak{X}_{1})$
and $\mathcal{S}_{1}\cap\mathcal{S}_{1}'\notin\Phi^{i}(\mathfrak{X}_{1})$
by the induction hypothesis; as $C\in\mathcal{S}_{1}\cap\mathcal{S}_{1}'$,
we get $\mathcal{T}_{1}=\left(\mathcal{S}_{1}\cup\mathcal{S}_{1}'\right)\backslash\{C\}\in\mathrm{ex}(\Phi^{i}(\mathfrak{X}_{1}))\subseteq\Phi^{i+1}(\mathfrak{X}_{1})$.

Conversely, suppose $\mathcal{A}_{l}\in\Phi^{i+1}(\mathfrak{X}_{l})$
for some $l$, say $l=1$, and let $\mathcal{T}\in\mathrm{ex}(\Phi^{i}(\mathfrak{X}_{1}))$
with $\mathcal{T}\subseteq\mathcal{A}_{1}$. If $\mathcal{T}\in\Phi^{i}(\mathfrak{X}_{1})$,
then $\mathcal{T}\in\Phi^{i}(\mathfrak{X})$ by the induction hypothesis.
Otherwise $\mathcal{T}=(\mathcal{S}\cup\mathcal{S}')\backslash\{C\}$
with $\mathcal{S},\mathcal{S}'\in\Phi^{i}(\mathfrak{X}_{1})$, $\mathcal{S}\cap\mathcal{S}'\notin\Phi^{i}(\mathfrak{X}_{1})$,
and $C\in\mathcal{S}\cap\mathcal{S}'$; then $\mathcal{S},\mathcal{S}'\in\Phi^{i}(\mathfrak{X})$
by the induction hypothesis, and likewise $\mathcal{S}\cap\mathcal{S}'\notin\Phi^{i}(\mathfrak{X})$,
since its first part is $\mathcal{S}\cap\mathcal{S}'$ itself and
its second part is $\ensuremath{\emptyset\notin\Phi^{i}(\mathfrak{X}_{2})}$.
In either case, $\mathcal{A}$ includes $\mathcal{T}\in\mathrm{ex}(\Phi^{i}(\mathfrak{X}))$,
so $\mathcal{A}\in\Phi^{i+1}(\mathfrak{X})$.

Taking the union over $i$ completes the proof.
\end{proof}

\begin{prop}
\label{prop:Syz-connected}Let $M$ have no coloops. Then $\mathrm{Syz}(M)$
is connected if and only if $M$ is connected.
\end{prop}

\begin{proof}
Suppose that $M$ is connected, and write $C\sim C'$ if the circuits
$C$ and $C'$ of $M$ lie in a common connected component of $\mathrm{Syz}(M)$.
We show that $C\sim C'$ for all circuits $C,C'\in\mathcal{C}(M)$.

Let $B$ be a base of $M$, and let $C_{1},C_{2}\in\mathcal{C}_{B}$
be distinct with $C_{1}\cap C_{2}\neq\emptyset$. Write $C_{i}=C(e_{i},B)$
with $e_{i}\in E(M)\backslash B$. Then $\mathrm{nullity}(C_{1}\cup C_{2})=2$
by Lemma \ref{lem:two-circ}, so Lemma \ref{lem:sp-two} shows that
$\mathrm{sp}^{\infty}(\{C_{1},C_{2}\})=\mathcal{C}(M|(C_{1}\cup C_{2}))$
and that this collection contains a circuit $C_{3}$ other than $C_{1}$
and $C_{2}$. By Propositions \ref{prop:syz-sn} and \ref{prop:simple},
the system $\{C_{1},C_{2},C_{3}\}$ has rank $2$ and is a circuit
of $\mathrm{Syz}(M)$. Therefore $C_{1}\sim C_{2}$. 

The fundamental graph $G_{B}(M)$ is the bipartite graph with parts
$B$ and $E(M)\backslash B$ in which $b$ and $e$ are adjacent if
and only if $b\in C(e,B)$. It is connected since $M$ is connected
\cite[Proposition 4.3.2]{Oxley}. Since $\mathcal{C}_{B}$ is indexed
by $E(M)\backslash B$, a path in $G_{B}(M)$ joining two such indices
yields a sequence of members of $\mathcal{C}_{B}$ in which consecutive
ones meet, so $C\sim C'$ for all $C,C'\in\mathcal{C}_{B}$ by the
previous paragraph. As $B$ is arbitrary, the same holds for all bases
of $M$.

Let $B'$ be another base. There is a sequence of bases $B_{0},B_{1},\dots,B_{t}$
of $M$ such that $B_{0}=B$, $B_{t}=B'$, and $\left|B_{i}\bigtriangleup B_{i-1}\right|=2$
for all $i\in[t]$. Putting $\{b\}=B_{i}\backslash B_{i-1}$ and $\{b'\}=B_{i-1}\backslash B_{i}$,
we have $C(b,B_{i-1})=C(b',B_{i})$, so $\mathcal{C}_{B_{i}}\cap\mathcal{C}_{B_{i-1}}\neq\emptyset$
and $C\sim C'$ for all $C\in\mathcal{C}_{B}$ and $C'\in\mathcal{C}_{B'}$.
Since every circuit of $M$ is a fundamental circuit, we have $C\sim C'$
for all $C,C'\in\mathcal{C}(M)$. Thus $\mathrm{Syz}(M)$ is connected. 

Conversely, suppose that $M$ is disconnected, say $M=M_{1}\oplus M_{2}$
with $E(M_{1})$ and $E(M_{2})$ nonempty. Since $M$ has no coloops,
$\mathcal{C}(M_{l})\neq\emptyset$ for $l=1,2$. Corollary \ref{cor:Syz-split}
then shows that $\mathrm{Syz}(M)=\mathrm{Syz}(M_{1})\oplus\mathrm{Syz}(M_{2})$
is disconnected.
\end{proof}

\begin{lem}
\label{lem:sn2-base}Let $\mathcal{A}$ be a circuit system of $M$
with $\left|\mathcal{A}\right|=3$ and $\mathrm{sn}(\mathcal{A})\le2$.
Then $\bigcup\mathcal{A}\subseteq B\cup\{e_{1},e_{2}\}$ for some
base $B$ of $M$ and distinct $e_{1},e_{2}\in E(M)\backslash B$.
\end{lem}

\begin{proof}
Let $\mathcal{A}'$ be a circuit system of least cardinality that
Gaussian spans $\mathcal{A}$, so that $|\mathcal{A}'|=\mathrm{sn}(\mathcal{A})\leq2$.
Since $\mathcal{A}\subseteq\mathrm{sp}^{\infty}(\mathcal{A}')$ and
$|\mathcal{A}|=3$, we have $\mathrm{sp}^{\infty}(\mathcal{A}')\neq\mathcal{A}'$.
As $\mathrm{sp}^{\infty}(\mathcal{A}')=\mathcal{A}'$ whenever $|\mathcal{A}'|\leq1$,
it follows that $\mathcal{A}'=\{D_{1},D_{2}\}$ for distinct circuits
$D_{1}$ and $D_{2}$; put $D=D_{1}\cup D_{2}$. Then Lemma \ref{lem:sp-two}
shows that $\mathrm{nullity}(D)=2$, $D_{1}\cap D_{2}\neq\emptyset$,
and $\mathrm{sp}^{\infty}(\mathcal{A}')=\mathcal{C}(M|D)$. Hence
Lemma \ref{lem:two-circ} gives a base $B$ of $M$ and $e_{1},e_{2}\in E(M)\backslash B$
such that $D_{l}=C(e_{l},B)$ for $l=1,2$. Since $\mathcal{A}\subseteq\mathcal{C}(M|D)$,
we get $\bigcup\mathcal{A}\subseteq D\subseteq B\cup\{e_{1},e_{2}\}$.
\end{proof}

\begin{lem}
\label{lem:size3}Every member of $\Phi^{\infty}(\mathfrak{X})$ of
cardinality $3$ lies in $\mathfrak{X}$. 
\end{lem}

\begin{proof}
Since $\Phi^{\infty}(\mathfrak{X})=\bigcup_{i\ge0}\Phi^{i}(\mathfrak{X})$,
it suffices to show that every member of $\Phi^{i}(\mathfrak{X})$
of cardinality $3$ lies in $\mathfrak{X}$, which we prove by induction
on $i$. The case $i=0$ is immediate. Write $\mathfrak{Y}=\Phi^{i}(\mathfrak{X})$
and assume the assertion for $\mathfrak{Y}$. Since $\Phi(\mathfrak{Y})\subseteq\Phi^{\infty}(\mathfrak{X})$
and $\mathrm{Syz}(M)$ is simple by Proposition \ref{prop:simple},
we have $\mathrm{girth}(\Phi(\mathfrak{Y}))\ge3$. As every member
of $\Phi(\mathfrak{Y})$ includes a member of $\mathrm{ex}(\mathfrak{Y})\subseteq\Phi(\mathfrak{Y})$,
each member of $\Phi(\mathfrak{Y})$ of cardinality $3$ lies in $\mathrm{ex}(\mathfrak{Y})$.
Members of $\mathfrak{Y}$ are covered by the induction hypothesis,
so we need only treat $(\mathcal{A}_{1}\cup\mathcal{A}_{2})\backslash\{C\}$
of cardinality $3$ with $\mathcal{A}_{1},\mathcal{A}_{2}\in\mathfrak{Y}$,
$\mathcal{A}_{1}\cap\mathcal{A}_{2}\notin\mathfrak{Y}$, and $C\in\mathcal{A}_{1}\cap\mathcal{A}_{2}$.

Since $\mathcal{A}_{1}\cup\mathcal{A}_{2}$ contains $C$, it has
cardinality $4$. Neither $\mathcal{A}_{1}$ nor $\mathcal{A}_{2}$
includes the other, as otherwise $\mathcal{A}_{1}\cap\mathcal{A}_{2}\in\mathfrak{Y}$,
so each is a proper subset of $\mathcal{A}_{1}\cup\mathcal{A}_{2}$
and $\left|\mathcal{A}_{1}\right|=\left|\mathcal{A}_{2}\right|=3$
by the girth bound. Hence $\mathcal{A}_{1},\mathcal{A}_{2}\in\mathfrak{X}$
by the induction hypothesis. Write $\mathcal{A}_{1}=\{C_{1},C_{2},C_{3}\}$
and $\mathcal{A}_{2}=\{C_{1},C_{2},C_{4}\}$ with $C_{3}\neq C_{4}$,
so that $C\in\{C_{1},C_{2}\}$. 

Since $\mathcal{A}_{1}\in\mathfrak{X}$ and $\left|\mathcal{A}_{1}\right|=3$,
by Lemma \ref{lem:sn2-base}, there exist a base $B$ of $M$ and
distinct $e_{1},e_{2}\in E(M)\backslash B$ such that $C_{1}$, $C_{2}$,
and $C_{3}$ are included in $\tilde{B}:=B\cup\{e_{1},e_{2}\}$. Hence
$C_{3}\subseteq C_{1}\cup C_{2}$ by Lemma \ref{lem:fund-circ}(\ref{enu:fund2}),
and the same argument applied to $\mathcal{A}_{2}$ gives $C_{4}\subseteq C_{1}\cup C_{2}$.
Now, $C_{3}$ is a circuit included in $\tilde{B}$ and distinct from
$C_{1}$ and $C_{2}$, and so is $C_{4}$. By Lemma \ref{lem:fund-circ}(\ref{enu:fund2})
and (\ref{enu:fund1}), each is $\left\lfloor (C_{1}\cup C_{2})\backslash\{c\}\right\rfloor $
for some $c\in C_{1}\cap C_{2}$ and hence lies in $\mathrm{sp}(\{C_{1},C_{2}\})$.
Therefore $\mathrm{sn}(\{C_{l},C_{3},C_{4}\})\le2$ for $l\in\{1,2\}$,
and $(\mathcal{A}_{1}\cup\mathcal{A}_{2})\backslash\{C\}\in\mathfrak{X}$. 
\end{proof}

\begin{cor}
\label{cor:rank-3}If $n-k\ge3$, then $r(\mathrm{Syz}(M))\ge3$.
\end{cor}

\begin{proof}
Let $B$ be a base of $M$ and take distinct $e_{1},e_{2},e_{3}\in E(M)\backslash B$,
which exist since $n-k\ge3$. Let $C_{i}=C(e_{i},B)$ for $i\in[3]$
and $\mathcal{A}=\{C_{1},C_{2},C_{3}\}$. Suppose $\mathrm{sn}(\mathcal{A})\le2$.
By Lemma \ref{lem:sn2-base}, the members of $\mathcal{A}$ are included
in $B'\cup\{f_{1},f_{2}\}$ for some base $B'$ and distinct $f_{1},f_{2}\in E(M)\backslash B'$.
Since $C_{1}$ is included in that set, $e_{1}$ lies in it, while
$C_{2}$ and $C_{3}$ both avoid $e_{1}$; this contradicts Lemma
\ref{lem:fund-circ}(\ref{enu:fund1}). Hence $\mathrm{sn}(\mathcal{A})=3$
by Proposition \ref{prop:sn-basic}(\ref{enu:sn-basic1}), so $\mathcal{A}\notin\mathfrak{X}$,
and $\mathcal{A}\notin\Phi^{\infty}(\mathfrak{X})$ by Lemma \ref{lem:size3}.
Thus $\mathcal{A}$ is independent in $\mathrm{Syz}(M)$ and $r(\mathrm{Syz}(M))\ge\left|\mathcal{A}\right|=3$. 
\end{proof}

\begin{cor}
\label{cor:sn-3-rank}If $M$ is connected and $\mathrm{sn}(\mathcal{C}(M))\le3$,
then $r(\mathrm{Syz}(M))=\mathrm{sn}(\mathcal{C}(M))$. 
\end{cor}

\begin{proof}
Put $s=\mathrm{sn}(\mathcal{C}(M))$. Since $r(\mathrm{Syz}(M))\le s$
by Corollary \ref{cor:rank-sn}, we prove the reverse inequality.
If $s=0$, then $\mathcal{C}(M)=\emptyset$ and both sides are $0$.
If $s=1$, then $n-k=1$ by Lemma \ref{lem:sn-12}, so $\mathcal{C}(M)$
has a single member and $r(\mathrm{Syz}(M))=1$. If $s=2$, then $\mathrm{Syz}(M)$
is simple by Proposition \ref{prop:simple}, and its ground set $\mathcal{C}(M)$
has at least two members by Proposition \ref{prop:sn-basic}(\ref{enu:sn-basic1}),
so $r(\mathrm{Syz}(M))\ge2$. If $s=3$, then $n-k\ge3$ by Corollary
\ref{cor:sn-3}, so $r(\mathrm{Syz}(M))\ge3$ by Corollary \ref{cor:rank-3}. 
\end{proof}

\begin{lem}
\label{lem:cl-two} For distinct circuits $C_{1}$ and $C_{2}$ of
$M$, 
\[
\mathrm{cl}_{\mathrm{Syz}(M)}(\{C_{1},C_{2}\})=\mathrm{sp}^{\infty}(\{C_{1},C_{2}\}).
\]
\end{lem}

\begin{proof}
Put $\mathcal{A}=\{C_{1},C_{2}\}$ and write $\mathrm{cl}$ for $\mathrm{cl}_{\mathrm{Syz}(M)}$.
Since $\mathrm{Syz}(M)$ is simple by Proposition \ref{prop:simple},
$\mathcal{A}$ is independent. Both $\mathrm{cl}(\mathcal{A})$ and
$\mathrm{sp}^{\infty}(\mathcal{A})$ include $\mathcal{A}$, so it
suffices to show, for a circuit $C\notin\mathcal{A}$, that $\mathcal{A}\cup\{C\}\in\Phi^{\infty}(\mathfrak{X})$
if and only if $C\in\mathrm{sp}^{\infty}(\mathcal{A})$. 

If $C\in\mathrm{sp}^{\infty}(\mathcal{A})$, then $\mathcal{A}$ Gaussian
spans $\mathcal{A}\cup\{C\}$, so $\mathrm{sn}(\mathcal{A}\cup\{C\})\le2<3$
and $\mathcal{A}\cup\{C\}\in\mathfrak{X}\subseteq\Phi^{\infty}(\mathfrak{X})$.
Conversely, if $\mathcal{A}\cup\{C\}\in\Phi^{\infty}(\mathfrak{X})$,
then $\mathcal{A}\cup\{C\}\in\mathfrak{X}$ by Lemma \ref{lem:size3},
so Lemma \ref{lem:sn2-base} gives a base $B$ of $M$ and $e_{1},e_{2}\in E(M)\backslash B$
such that $C_{1}$, $C_{2}$, and $C$ are included in $\tilde{B}:=B\cup\{e_{1},e_{2}\}$.
Lemma \ref{lem:fund-circ}(\ref{enu:fund2}) and (\ref{enu:fund1})
then show that $C\subseteq C_{1}\cup C_{2}$ and $C=\left\lfloor (C_{1}\cup C_{2})\backslash\{c\}\right\rfloor $
for some $c\in C_{1}\cap C_{2}$, and hence $C\in\mathrm{sp}(\mathcal{A})\subseteq\mathrm{sp}^{\infty}(\mathcal{A})$.
\end{proof}

\begin{cor}
\label{cor:rank2-flats}Let $C_{1}$ and $C_{2}$ be distinct circuits
of $M$ and put $D=C_{1}\cup C_{2}$. If $\mathrm{nullity}(D)=2$,
then $\mathcal{C}(M|D)$ is a rank-$2$ flat of $\mathrm{Syz}(M)$;
otherwise $\{C_{1},C_{2}\}$ is. Every rank-$2$ flat of $\mathrm{Syz}(M)$
arises in this way.
\end{cor}

\begin{proof}
By Lemmas \ref{lem:cl-two} and \ref{lem:sp-two}, we have $\mathrm{cl}_{\mathrm{Syz}(M)}(\{C_{1},C_{2}\})=\mathrm{sp}^{\infty}(\{C_{1},C_{2}\})$,
which is $\mathcal{C}(M|D)$ if $\mathrm{nullity}(D)=2$ and $\{C_{1},C_{2}\}$
otherwise. Since $\{C_{1},C_{2}\}$ is independent in $\mathrm{Syz}(M)$
by Proposition \ref{prop:simple}, this closure is a rank-$2$ flat.
Conversely, let $\mathcal{A}$ be a rank-$2$ flat of $\mathrm{Syz}(M)$.
Then $\mathcal{A}$ includes an independent pair $\{C_{1},C_{2}\}$,
hence includes its closure, and the two are equal since both are rank-$2$
flats.
\end{proof}

Since $\mathrm{sp}^{\infty}$ is not a matroid closure operator by
Example \ref{exa:Vamos1}(\ref{enu:not-mcl}), while $\mathrm{cl}_{\mathrm{Syz}(M)}$
is, the equality in Lemma \ref{lem:cl-two} cannot hold for all circuit
systems. For independent circuit systems, only the inclusion of Corollary
\ref{cor:sp-cl} is available in general.\medskip{}

In the next two examples $\mathrm{Syz}(M)$ has rank $3$. Following
viewpoint \ref{enu:HA-viewpoint}, we call the mp-flats of $\mathrm{Syz}(M)$
\emph{points} and the corank-$2$ flats \emph{lines}. Any two distinct
lines then meet in a point.
\begin{example}[Uniform matroids $U_{k,k+3}$]
 Let $k\ge1$ and put $M=U_{k,k+3}$. Then $M$ has $\binom{k+3}{2}$
circuits, each of size $k+1$. Since $n-k=3$, we have $r(\mathrm{Syz}(M))=3$
by Corollaries \ref{cor:rank-sn} and \ref{cor:rank-3}, and the singletons
$\left\{ C\right\} $ are rank-$1$ flats by Proposition \ref{prop:simple}.
Let $C_{1}$ and $C_{2}$ be distinct circuits of $M$. Then $r(C_{1}\cup C_{2})=k$
since $\left|C_{1}\cup C_{2}\right|\ge k+2$. By Corollary \ref{cor:rank2-flats},
the rank-$2$ flats of $\mathrm{Syz}(M)$ are described as follows.
\begin{itemize}
\item If $\left|C_{1}\cup C_{2}\right|=k+2$, then $\mathrm{nullity}(C_{1}\cup C_{2})=2$
and $C_{1}\cup C_{2}=E(M)\backslash\{e\}$ for some $e$, so $\mathcal{C}(M\backslash\{e\})$
is a rank-$2$ flat, of cardinality $k+2$ since $M\backslash\{e\}\simeq U_{k,k+2}$.
There are $k+3$ of them.
\item If $\left|C_{1}\cup C_{2}\right|=k+3$, then $\mathrm{nullity}(C_{1}\cup C_{2})=3$,
and $\{C_{1},C_{2}\}$ is a rank-$2$ flat. There are $3\binom{k+3}{4}$
of them, one for each partition of a four-element subset of $E(M)$
into two pairs.
\end{itemize}
Each circuit $C$ equals $E(M)\backslash\{e_{1},e_{2}\}$ for some
$e_{1}\neq e_{2}$, and $\{C\}=\mathcal{C}(M\backslash\{e_{1}\})\cap\mathcal{C}(M\backslash\{e_{2}\})$.
Hence no three of the $k+3$ points of the form $\mathcal{C}(M\backslash\{e\})$
are collinear, and each line passes through exactly two of them. Figure
\ref{fig:Syz-U25} shows the case $k=2$. 
\begin{figure}[h]
\begin{spacing}{0.4}
\begin{centering}
\begin{center}
\begin{tikzpicture}[font=\scriptsize]

\begin{scope}[rotate=-5,xshift=0cm,yshift=0cm,line cap=round,scale=0.44]
  \foreach \i in {1,...,5}{
    \path (0,0)--++({90+72*(\i-1)}:4) coordinate (A\i);
  }
  \foreach \i in {1,...,5}{
    \path (A\i)++({108+72*(\i-1)}:1.6) coordinate (E\i);
  }

 \draw [name path=L1][thick, shorten >=-2em, shorten <=-2em] (E1)--(A4);
 \draw [name path=L2][thick, shorten >=-2em, shorten <=-2em] (E2)--(A5);
 \draw [name path=L3][thick, shorten >=-2em, shorten <=-2em] (E3)--(A1);
 \draw [name path=L4][thick, shorten >=-2em, shorten <=-2em] (E4)--(A2);
 \draw [name path=L5][thick, shorten >=-2em, shorten <=-2em] (E5)--(A3);

 \path [name intersections={of=L1 and L2, by=B5}];
 \path [name intersections={of=L2 and L3, by=B1}];
 \path [name intersections={of=L3 and L4, by=B2}];
 \path [name intersections={of=L4 and L5, by=B3}];
 \path [name intersections={of=L5 and L1, by=B4}];

 \draw [name path=N1][very thick,cyan, shorten >=-2em, shorten <=-2em] (B3)--(B1) .. controls ++(108:1.6) .. (E1);
 \draw [name path=N2][very thick,blue, shorten >=-2em, shorten <=-2em] (B4)--(B2) .. controls +(180:1.6) .. (E2);
 \draw [name path=N3][very thick,green, shorten >=-2em, shorten <=-2em] (B5)--(B3) .. controls +(252:1.6) .. (E3);
 \draw [name path=N4][very thick,orange, shorten >=-2em, shorten <=-2em] (B1)--(B4) .. controls +(324:1.6) .. (E4);
 \draw [name path=N5][very thick,red, shorten >=-2em, shorten <=-2em] (B2)--(B5) .. controls +(396:1.6) .. (E5);

 \path [name intersections={of=N1 and N5, by=C1}];
 \path [name intersections={of=N2 and N1, by=C2}];
 \path [name intersections={of=N3 and N2, by=C3}];
 \path [name intersections={of=N4 and N3, by=C4}];
 \path [name intersections={of=N5 and N4, by=C5}];

  \foreach \i in {1,...,5}{
    \foreach \j in {A,E,B,C}{
     \fill (\j\i) circle (4.5pt);
    }
  }

  \foreach \i in {1,...,5}{
    \fill (B\i) circle (6pt);
  }
\path (E1)++(90:1.5) (E3)++(270:1.5); 

\end{scope}

\end{tikzpicture}
\par\end{center}
\par\end{centering}
\end{spacing}
\begin{spacing}{0.3}
\noindent\caption{\label{fig:Syz-U25}The syzygy matroid $\mathrm{Syz}(U_{2,5})$.}
\end{spacing}
\end{figure}
\end{example}

\begin{example}[Vámos matroid, continued]
\label{exa:Vamos2} Let $M=V_{8}$ be the Vámos matroid as in Example
\ref{exa:Vamos1}. Since $\mathrm{sn}(\mathcal{C}(M))=3$ by Example
\ref{exa:Vamos1}(\ref{enu:Vamos-sn-3}), we have $r(\mathrm{Syz}(M))=3$
by Corollary \ref{cor:sn-3-rank}. Let $D_{1}$ and $D_{2}$ be distinct
circuits of $M$. Then $\left|D_{1}\cup D_{2}\right|\ge6$ and $r(D_{1}\cup D_{2})=4$. 
\begin{itemize}
\item If $\left|D_{1}\cup D_{2}\right|=6$, then $\mathrm{nullity}(D_{1}\cup D_{2})=2$
and $D_{1}\cup D_{2}=E(M)\backslash\{e_{1},e_{2}\}$ for some $e_{1}\neq e_{2}$,
so $\mathcal{C}(M\backslash\{e_{1},e_{2}\})$ is a rank-$2$ flat,
of cardinality at least $3$ by Lemma \ref{lem:sp-two}, since $D_{1}\cap D_{2}\neq\emptyset$.
There are $28$ such flats, one for each pair $\{e_{1},e_{2}\}$,
and a circuit $C$ is contained in $\mathcal{C}(M\backslash\{e_{1},e_{2}\})$
if and only if $C\cap\{e_{1},e_{2}\}=\emptyset$. Hence each of $C_{0},\dots,C_{4}$
is contained in six of these flats and each five-element circuit in
three. Let $t$ be the number of circuits among $C_{0},\ldots,C_{4}$
that avoid $\{e_{1},e_{2}\}$. Each of them is included in exactly
two of the six five-element subsets of $E(M)\backslash\{e_{1},e_{2}\}$,
and no two of them in the same one, so by Example \ref{exa:Vamos1}(\ref{enu:Vamos-DEF})
the flat $\mathcal{C}(M\backslash\{e_{1},e_{2}\})$ has $t+(6-2t)=6-t$
members. These $28$ flats fall into five types, listed below with
$t=3,2,1,1,0$, respectively.
\begin{enumerate}[topsep=6pt]
\item \label{enu:Vamos2-1} $\mathcal{C}(M\backslash56)$, $\mathcal{C}(M\backslash78)$:
cardinality $3$, consisting of $C_{0}$ and two of $C_{1},\dots,C_{4}$.
\item \label{enu:Vamos2-2}$\mathcal{C}(M\backslash12)$, $\mathcal{C}(M\backslash34)$:
cardinality $4$, containing two of $C_{1},\dots,C_{4}$.
\item \label{enu:Vamos2-3}$\mathcal{C}(M\backslash57)$, $\mathcal{C}(M\backslash58)$,
$\mathcal{C}(M\backslash67)$, $\mathcal{C}(M\backslash68)$: cardinality
$5$, containing $C_{0}$. 
\item \label{enu:Vamos2-4}$\mathcal{C}(M\backslash ij)$ with $i\le4<j$:
cardinality $5$, sixteen of them, each containing one of $C_{1},\dots,C_{4}$. 
\item \label{enu:Vamos2-5}$\mathcal{C}(M\backslash13)$, $\mathcal{C}(M\backslash14)$,
$\mathcal{C}(M\backslash23)$, $\mathcal{C}(M\backslash24)$: cardinality
$6$, consisting only of five-element circuits.
\end{enumerate}
\item If $\left|D_{1}\cup D_{2}\right|\ge7$, then $\mathrm{nullity}(D_{1}\cup D_{2})\ge3$
and $\{D_{1},D_{2}\}$ is a flat of rank $2$. There are $\binom{41}{2}=820$
pairs of distinct circuits of $M$, of which $2\binom{3}{2}+2\binom{4}{2}+4\binom{5}{2}+16\binom{5}{2}+4\binom{6}{2}=278$
are included in one of the $28$ flats above. The remaining $542$
pairs are the rank-$2$ flats arising in this case.
\end{itemize}
There are $570$ rank-$2$ flats in total. Figure \ref{fig:Syz-Vamos}
shows the five lines $\left\{ C_{0}\right\} ,\dots,\left\{ C_{4}\right\} $
together with four other lines $\left\{ i5678\right\} $ for $i\in\left[4\right]$
and their intersection points. 
\begin{figure}[h]
\begin{spacing}{0.4}
\begin{centering}
\begin{center}
\begin{tikzpicture}[font=\scriptsize]

\begin{scope}[xshift=0cm,yshift=0cm,line cap=round,scale=0.7]
  \foreach \i in {1,...,3}{
    \path (0,0)--++({90+120*(\i-1)}:4) coordinate (A\i);
  }

 \draw [name path=L0][thick] (A2)++(180:4.5)--(A3)--++(0:4.5);
 \path ($(A2)!.5!(A3)$) node[below=-1pt]{$\{C_0\}$};
 \draw [name path=L1][thick] (A1)++(60:1)--(A2)--++(240:1)node[right=-2pt]{$\{C_1\}$};
 \draw [name path=L2][thick] (A1)++(120:1)--(A3)--++(300:1)node[left=-2pt]{$\{C_2\}$};
 \draw [name path=L3][thick] (A1)++(70:1)--++(250:8.3)node[right=-2pt]{$\{15678\}$};
 \draw [name path=L4][thick] (A1)++(110:1)--++(290:8.3)node[left=-2pt]{$\{25678\}$};

 \draw [name path=N1][thick] (A2)++(220:1)node[left=-3pt]{$\{C_3\}$}--++(40:8);
 \draw [name path=N2][thick] (A3)++(320:1)node[right=-3pt]{$\{C_4\}$}--++(140:8);
 \path [name intersections={of=N1 and N2, by=A0}];
 \draw [name path=N3][thick] (A0)++(22:2.5)--++(202:11)node[below=-2pt]{$\{35678\}$};
 \draw [name path=N4][thick] (A0)++(158:2.5)--++(338:11)node[below=-2pt]{$\{45678\}$};

 \path [name intersections={of=L0 and N3, by=B0}];
 \path [name intersections={of=L0 and N4, by=C0}];
 \path [name intersections={of=L0 and L3, by=D0}];
 \path [name intersections={of=L0 and L4, by=E0}];

 \path [name intersections={of=L1 and N2, by=B1}];
 \path [name intersections={of=L1 and N3, by=B2}];
 \path [name intersections={of=L1 and N4, by=B3}];

 \path [name intersections={of=L2 and N1, by=C1}];
 \path [name intersections={of=L2 and N3, by=C2}];
 \path [name intersections={of=L2 and N4, by=C3}];

 \path [name intersections={of=L3 and N1, by=D1}];
 \path [name intersections={of=L3 and N2, by=D2}];
 \path [name intersections={of=L3 and N3, by=D3}];
 \path [name intersections={of=L3 and N4, by=D4}];

 \path [name intersections={of=L4 and N1, by=E1}];
 \path [name intersections={of=L4 and N2, by=E2}];
 \path [name intersections={of=L4 and N3, by=E3}];
 \path [name intersections={of=L4 and N4, by=E4}];

  \foreach \i in {1,...,4}{
     \draw [thick,shorten >=-5pt, shorten <=-5pt] (D4)--++({250-17.6*\i}:2pt);
  }
  \foreach \i in {1,...,4}{
     \draw [thick,shorten >=-5pt, shorten <=-5pt] (E3)--++({290+17.6*\i}:2pt);
  }
  \foreach \i in {1,...,4}{
     \draw [thick,shorten >=-10pt, shorten <=-10pt] (D3)--++({70-9.6*\i}:-3pt);
  }
  \foreach \i in {1,...,4}{
     \draw [thick,shorten >=-10pt, shorten <=-10pt] (E4)--++({110+9.6*\i}:-3pt);
  }

\draw [->,shorten >=3pt] ($(D3)+(150:1.5)$) .. controls +(-20:0.6) and +(120:0.6) .. (D3)
node [at start,left] {$\mathcal{C}(M\backslash24)$};
\draw [->,shorten >=3pt] ($(D4)+(150:2.5)$) .. controls +(-20:0.6) and +(120:0.6) .. (D4)
node [at start,left] {$\mathcal{C}(M\backslash23)$};

 \draw [->,shorten >=3pt] ($(E3)+(30:2.5)$) .. controls +(200:0.6) and +(60:0.6) .. (E3)
  node [at start,right] {$\mathcal{C}(M\backslash14)$};
 \draw [->,shorten >=3pt] ($(E4)+(30:1.5)$) .. controls +(200:0.6) and +(60:0.6) .. (E4)
  node [at start,right] {$\mathcal{C}(M\backslash13)$};

\draw [->,shorten >=3pt] ($(A0)+(270:1)$) .. controls +(80:0.4) and +(250:0.4) .. (A0)
node [at start,below] {$\mathcal{C}(M\backslash12)$};
\draw [->,shorten >=3pt] ($(A1)+(150:1.2)$) .. controls +(-20:0.4) and +(160:0.4) .. (A1)
node [at start,left] {$\mathcal{C}(M\backslash34)$};

\draw [->,shorten >=3pt] ($(A2)+(150:0.6)$) .. controls +(-20:0.3) and +(120:0.3) .. (A2)
node [at start,left] {$\mathcal{C}(M\backslash78)$};
 \draw [->,shorten >=3pt] ($(A3)+(30:0.6)$) .. controls +(200:0.3) and +(60:0.3) .. (A3)
  node [at start,right] {$\mathcal{C}(M\backslash56)$};

  \foreach \i in {0,...,3}{
    \foreach \j in {A,B,C}{
     \fill (\j\i) circle (3pt);
    }
  }
  \foreach \i in {0,...,4}{
    \foreach \j in {D,E}{
     \fill (\j\i) circle (3pt);
    }
  }

\end{scope}

\end{tikzpicture}
\par\end{center}
\par\end{centering}
\end{spacing}
\begin{spacing}{0.3}
\noindent\caption{\label{fig:Syz-Vamos}Nine lines of $\mathrm{Syz}(V_{8})$ and their
intersection points. Short segments indicate further lines through
a point.}
\end{spacing}
\end{figure}
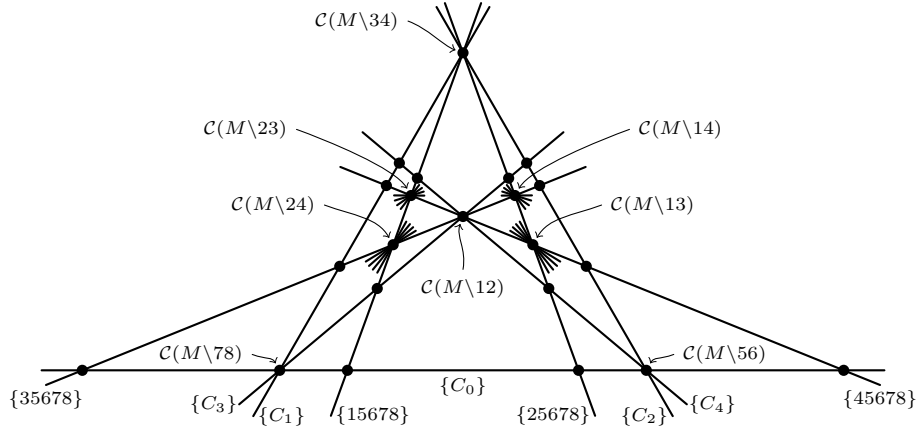
\end{example}

\section{\label{sec:Combinatorial-Adjoints}Combinatorial Adjoints }

Definition \ref{def:Comb-Adjoint} below makes sense for every matroid
$M$, with $\mathrm{Syz}(M^{\ast})$ as defined at the beginning of
Section \ref{sec:Syzygy}. A coloop $e$ of $M$ is a loop of $M^{\ast}$,
and under the map $\beta$ introduced below, the circuit $\left\{ e\right\} $
of $M^{\ast}$ corresponds to the mp-flat $E(M)\backslash\{e\}$ of
$M$. Put $p=\left|\mathrm{cl}_{M^{\ast}}(\emptyset)\right|$. Every
mp-flat of $M$ other than the $p$ mp-flats $E(M)\backslash\{e\}$
with $e\in\mathrm{cl}_{M^{\ast}}(\emptyset)$ includes $\mathrm{cl}_{M^{\ast}}(\emptyset)$,
and the map $H\mapsto H\backslash\mathrm{cl}_{M^{\ast}}(\emptyset)$
is a bijection from the collection of such mp-flats onto $\mathcal{H}(M\backslash\mathrm{cl}_{M^{\ast}}(\emptyset))$.
Hence 
\[
\mathrm{Adj}(M)\simeq\mathrm{Adj}(M\backslash\mathrm{cl}_{M^{\ast}}(\emptyset))\oplus U_{p,p},
\]
 where $U_{p,p}\simeq\bigoplus^{p}_{i=1}U_{1,1}$ is the free matroid
on the collection of these $p$ mp-flats. Accordingly, for the rest
of this section, $M$ is assumed to have no coloops unless stated
otherwise, so that $M^{\ast}$ is loopless and the results of Section
\ref{sec:Syzygy} apply to $M^{\ast}$. Loops of $M$ are allowed.
Proposition \ref{prop:Adj-simple} below, together with the simplicity
of $U_{p,p}$, shows that $\mathrm{Adj}(M)$ is also simple when $M$
has coloops.\medskip{}

Dualizing any system of matroid axioms yields another system of matroid
axioms. For example, the spanning-set axioms are obtained from the
independent-set axioms by duality: a set $S\subseteq E(M)$ spans
$M$ if and only if $E(M)\backslash S$ is independent in $M^{\ast}$.
We now dualize the dependent-set axioms.\medskip{}

A collection $\mathfrak{A}\subseteq2^{E}$ is the collection of complements
of the dependent sets of a matroid $N$ on $E$ if and only if
\begin{enumerate}[label=(A\arabic*),start=0]
\item \label{enu:nonspanning0}$E\notin\mathfrak{A}$. 
\item \label{enu:nonspanning1}If $A\in\mathfrak{A}$ and $A'\subseteq A$,
then $A'\in\mathfrak{A}$.
\item \label{enu:nonspanning2}If $A_{1},A_{2}\in\mathfrak{A}$, then either
$A_{1}\cup A_{2}\in\mathfrak{A}$, or $(A_{1}\cap A_{2})\cup\{e\}\in\mathfrak{A}$
for all $e\in E\backslash(A_{1}\cup A_{2})$.
\end{enumerate}
The maximal members of $\mathfrak{A}$ are the complements of the
circuits of $N$, that is, the mp-flats of $N^{\ast}$. For $N=M^{\ast}$,
these conditions characterize the collection of nonspanning sets of
$M$, whose maximal members are the mp-flats of $M$.\medskip{}

Define a map $\beta:2^{E(M)}\rightarrow2^{E(M)}$ by $X\mapsto E(M)\backslash X$,
and for a collection $\mathcal{A}\subseteq2^{E(M)}$, write $\beta\mathcal{A}:=\left\{ \beta(X):X\in\mathcal{A}\right\} $.
We write $\check{X}$ for $\beta(X)$ when $X$ is a single symbol
or a string of digits, as in $\check{C}$ and $\check{78}$, and keep
$\beta$ for collections and compound expressions, as in $\beta\mathcal{A}$
and $\beta(X\cup Y)$. For a seed $\mathfrak{X}$, we also write $\check{\mathfrak{X}}:=\{\beta\mathcal{A}:\mathcal{A}\in\mathfrak{X}\}$.
The map $\beta$ restricts to a bijection from $\mathcal{C}(M^{\ast})$
onto $\mathcal{H}(M)$, and it carries unions to intersections. Since
an intersection of flats is a flat, $\beta(\bigcup\mathcal{A})$ is
a flat of $M$ for $\mathcal{A}\subseteq\mathcal{C}(M^{\ast})$. The
identity $\mathrm{nullity}_{M^{\ast}}(X)=\mathrm{corank}_{M}(\check{X})$
holds for all $X\subseteq E(M)$.\medskip{}

Define a map $\psi:\mathcal{L}(M)\rightarrow2^{\mathcal{H}(M)}$ by
$F\mapsto\left\{ H\in\mathcal{H}(M):F\subseteq H\right\} $. Since
every flat of $M$ is the intersection of the mp-flats including it,
where the empty intersection is understood to be $E(M)$, the map
$\psi$ is injective and reverses inclusion: $F\subseteq F'$ if and
only if $\psi(F')\subseteq\psi(F)$. For $F\in\mathcal{L}(M)$, the
map $\beta$ restricts to a bijection from $\mathcal{C}(M^{\ast}|\check{F})$
onto $\psi(F)$. 
\begin{lem}
\label{lem:floor-cl}As maps on $2^{E(M)}$, we have $\beta\circ\left\lfloor \cdot\right\rfloor _{M^{\ast}}=\mathrm{cl}_{M}\circ\beta$. 
\end{lem}

\begin{proof}
Let $D\subseteq E(M)$. As $M$ has no coloops, $M^{\ast}$ is loopless,
so $\left\lfloor D\right\rfloor _{M^{\ast}}$ is the union of the
circuits of $M^{\ast}$ included in $D$. Since $\beta$ carries these
circuits to the mp-flats of $M$ including $\check{D}$, the set $\beta\left\lfloor D\right\rfloor _{M^{\ast}}$
is the intersection of those mp-flats, which is $\mathrm{cl}_{M}(\check{D})$
because a flat of $M$ is the intersection of the mp-flats including
it. 
\end{proof}

Through $\beta$, the elimination underlying the Gaussian operations
of Section \ref{sec:Syzygy} becomes the step appearing in \ref{enu:nonspanning2}:
if $D_{1}$ and $D_{2}$ are dependent in $M^{\ast}$ and $e\in D_{1}\cap D_{2}$,
then 
\begin{equation}
\beta\bigl((D_{1}\cup D_{2})\backslash\{e\}\bigr)=(\check{D}_{1}\cap\check{D}_{2})\cup\{e\},\label{eq:beta}
\end{equation}
 where $e\notin\check{D}_{1}\cup\check{D}_{2}$. Combined with Lemma
\ref{lem:floor-cl}, this transports the Gaussian process for $M^{\ast}$
to a process on the flats of $M$.\medskip{}

For a collection $\mathcal{Y}=\left\{ H_{1},\dots,H_{m}\right\} $
of mp-flats of $M$ and an mp-flat $H$, the dual Gaussian process
for $\left(\mathcal{Y},H\right)$ proceeds as the Gaussian process
of Section \ref{sec:Syzygy} does, with the following changes: $H$
plays the role of $C$; $F_{i}$ is initially $H_{i}$; the element
$c_{j}$ is chosen from $H\backslash F_{\lambda_{j}}$; and $F_{i}$
is replaced by $\mathrm{cl}_{M}\bigl((F_{i}\cap F_{\lambda_{j}})\cup\{c_{j}\}\bigr)$
for those $i$ with $c_{j}\notin F_{i}$. Since $\ensuremath{c_{j}\notin F_{\lambda_{j}}}$
and flats of $M$ contain the loops of $M$, $c_{j}$ is not a loop.
Hence each update joins the meet $F_{i}\cap F_{\lambda_{j}}$ with
the rank-$1$ flat $\mathrm{cl}_{M}(\{c_{j}\})$, which is not below
the meet, so the new $F_{i}$, again a flat of $M$, has corank one
less than that of the meet.
\begin{prop}
\label{prop:dual-Gaussian}Let $\mathcal{A}$ be a circuit system
of $M^{\ast}$, and let $C$ be a circuit of $M^{\ast}$. Then $\beta$
carries the Gaussian process for $\left(\mathcal{A},C\right)$, step
by step, to the dual Gaussian process for $\left(\beta\mathcal{A},\check{C}\right)$.
In particular, $C$ is a Gaussian combination of $\mathcal{A}$ if
and only if some execution of the dual Gaussian process for $\left(\beta\mathcal{A},\check{C}\right)$
succeeds.
\end{prop}

\begin{proof}
Apply $\beta$ to every set occurring in the Gaussian process for
$(\mathcal{A},C)$, keeping the indices and the choices, so that $\check{X}_{i}$
plays the role of $F_{i}$. The condition $c_{j}\in X_{\lambda_{j}}\backslash C$
becomes $c_{j}\in\check{C}\backslash\check{X}_{\lambda_{j}}$, and
the condition $c_{j}\in X_{i}$ becomes $c_{j}\notin\check{X}_{i}$.
In the latter case, $c_{j}\in X_{i}\cap X_{\lambda_{j}}$, so by (\ref{eq:beta})
and Lemma \ref{lem:floor-cl}, we have 
\[
\beta\lfloor(X_{i}\cup X_{\lambda_{j}})\backslash\{c_{j}\}\rfloor_{M^{*}}=\mathrm{cl}_{M}((\check{X}_{i}\cap\check{X}_{\lambda_{j}})\cup\{c_{j}\}),
\]
 which is the update in the dual Gaussian process. As $\beta$ is
injective, two sets coincide if and only if their images do, so the
stopping condition, success, and merges correspond. Hence the executions
of the two processes correspond, which gives the last assertion.
\end{proof}

We write $\check{\mathrm{sp}}$, $\check{\mathrm{sp}}^{\infty}$,
and $\check{\mathrm{sn}}$ for $\mathrm{sp}$, $\mathrm{sp}^{\infty}$,
and $\mathrm{sn}$ transported through $\beta$, so that $\check{\mathrm{sp}}(\beta\mathcal{A})=\beta\,\mathrm{sp}(\mathcal{A})$,
$\check{\mathrm{sp}}^{\infty}(\beta\mathcal{A})=\beta\,\mathrm{sp}^{\infty}(\mathcal{A})$,
and $\check{\mathrm{sn}}(\beta\mathcal{A})=\mathrm{sn}(\mathcal{A})$
for a circuit system $\mathcal{A}$ of $M^{\ast}$. By Proposition
\ref{prop:dual-Gaussian}, these can be computed within the flat lattice
of $M$. 
\begin{defn}
\label{def:Comb-Adjoint} The \emph{combinatorial adjoint} of a matroid
$M$ is the matroid $\mathrm{Adj}(M)$ on $\mathcal{H}(M)$ such that
$\beta$ restricts to an isomorphism from $\mathrm{Syz}(M^{\ast})$
onto $\mathrm{Adj}(M)$.
\end{defn}

We identify $\mathrm{Adj}(M)$ with $\mathrm{Syz}(M^{\ast})$ through
$\beta$ and work in whichever of the two is more convenient. Since
$\mathrm{uc}$ and $\mathrm{ex}$ are defined in terms of inclusion,
membership, union, intersection, and difference alone, they commute
with the maps induced by $\beta$ on collections, and hence 
\[
\mathrm{Adj}(M)=\mathrm{Gen}(\check{\mathfrak{X}}),\qquad\check{\mathfrak{X}}=\{\mathcal{Y}\subseteq\mathcal{H}(M):\left|\mathcal{Y}\right|>\check{\mathrm{sn}}(\mathcal{Y})\},
\]
 where $\check{\mathfrak{X}}$ is the image under $\beta$ of the
seed of $\mathrm{Syz}(M^{\ast})$ and $\mathrm{Gen}$ is taken with
$\mathcal{H}(M)$ as the ground set. Thus $\mathrm{Adj}(M)$ is constructed
entirely within the flat lattice of $M$. Unlike $\mathrm{Adj}(M,L)$
in Section \ref{sec:Realizations}, which depends on a realization
$L$, the matroid $\mathrm{Adj}(M)$ is determined by $M$ alone.
Fu, Tang, and Wang ask how to construct a matroid of this kind \cite[Section 7]{Adjoint}. 
\begin{rem}
The construction of Section \ref{sec:Generating-Matroids} also dualizes
under complementation within the ground set. Let $E$ be a finite
set and $\mathfrak{Y}\subseteq2^{E}$ with $E\notin\mathfrak{Y}$.
Replace upward closure by downward closure and the elimination extension
by the operator adjoining $(Y_{1}\cap Y_{2})\cup\{e\}$ whenever $Y_{1},Y_{2}\in\mathfrak{Y}$,
$Y_{1}\cup Y_{2}\notin\mathfrak{Y}$, and $e\in E\backslash(Y_{1}\cup Y_{2})$.
By \ref{enu:nonspanning0}–\ref{enu:nonspanning2}, iterating these
operators yields the nonspanning-set collection of the dual of the
matroid generated by $\left\{ E\backslash Y:Y\in\mathfrak{Y}\right\} $.
Applied to $\left\{ \mathcal{H}(M)\backslash\mathcal{Y}:\mathcal{Y}\in\check{\mathfrak{X}}\right\} $,
this procedure produces the nonspanning-set collection of $\mathrm{Adj}(M)^{\ast}$. 
\end{rem}

All results of Section \ref{sec:Syzygy} carry over through $\beta$,
and we restate only a few of them, beginning with fundamental circuit
systems. 
\begin{lem}
\label{lem:rank-sn-check}Let $B$ be a base of $M$. Then $\beta$
carries the fundamental circuit system of $\check{B}$ in $M^{\ast}$
onto $\mathcal{H}_{B}:=\{\mathrm{cl}_{M}(B\backslash\{e\}):e\in B\}$,
and $\check{\mathrm{sp}}^{\infty}(\mathcal{H}_{B})=\mathcal{H}(M)$. 
\end{lem}

\begin{proof}
Let $e\in B$. Both $\beta C_{M^{\ast}}(e,\check{B})$ and $\mathrm{cl}_{M}(B\backslash\{e\})$
are mp-flats of $M$ including the corank-$1$ set $B\backslash\{e\}$,
so they coincide. The second assertion is Lemma \ref{lem:sn-upper}
applied to $M^{\ast}$, under $\beta$.
\end{proof}

Under the complementation $\beta$, Corollaries \ref{cor:syz-sn-nullity}, \ref{cor:rank-3}, and \ref{cor:sn-3-rank}
and Propositions \ref{prop:simple} and \ref{prop:Syz-connected} become
the following. 
\begin{cor}
\label{cor:upper-bounds}Let $\mathcal{Y}\subseteq\mathcal{H}(M)$.
Then 
\[
r_{\mathrm{Adj}(M)}(\mathcal{Y})\le\check{\mathrm{sn}}(\mathcal{Y})\le\mathrm{corank}_{M}\bigl(\bigcap\mathcal{Y}\bigr).
\]
 In particular, $r(\mathrm{Adj}(M))\le\check{\mathrm{sn}}(\mathcal{H}(M))\le k$. 
\end{cor}

\begin{cor}
\label{cor:lower-bound}If $k\ge3$, then $r(\mathrm{Adj}(M))\ge3$. 
\end{cor}

\begin{cor}
\label{cor:sn-check-3-rank}If $M$ is connected and $\check{\mathrm{sn}}(\mathcal{H}(M))\le3$,
then $r(\mathrm{Adj}(M))=\check{\mathrm{sn}}(\mathcal{H}(M))$.
\end{cor}

\begin{prop}
\label{prop:Adj-simple}$\mathrm{Adj}(M)$ is simple.
\end{prop}

\begin{prop}
\label{prop:Adj-connected}Let $M$ be loopless. Then $\mathrm{Adj}(M)$
is connected if and only if $M$ is connected. 
\end{prop}

Propositions \ref{prop:Syz-connected} and \ref{prop:Adj-connected}
establish the connectedness statement of Section \ref{sec:Introduction}.
Indeed, a connected matroid with at least two elements has neither
loops nor coloops. If $M$ has at most one element, then so do $\mathrm{Syz}(M)$
and $\mathrm{Adj}(M)$, and matroids with at most one element are
connected by convention.
\begin{cor}
Let $M=M_{1}\oplus M_{2}$. Then $\mathrm{Adj}(M)\simeq\mathrm{Adj}(M_{1})\oplus\mathrm{Adj}(M_{2})$. 
\end{cor}

\begin{proof}
A flat of $M$ is a union of a flat of $M_{1}$ and a flat of $M_{2}$,
and ranks add, so the mp-flats of $M$ are the sets $H\cup E(M_{2})$
with $H\in\mathcal{H}(M_{1})$ and $E(M_{1})\cup H$ with $H\in\mathcal{H}(M_{2})$.
Under $\beta$, this splitting corresponds to $\mathcal{C}(M^{\ast})=\mathcal{C}(M^{\ast}_{1})\sqcup\mathcal{C}(M^{\ast}_{2})$,
since $M^{\ast}=M^{\ast}_{1}\oplus M^{\ast}_{2}$. Now apply Corollary
\ref{cor:Syz-split} to $M^{\ast}$, noting that $M_{1}$ and $M_{2}$
have no coloops.
\end{proof}

\begin{lem}
For distinct mp-flats $H_{1}$ and $H_{2}$ of $M$, 
\[
\mathrm{cl}_{\mathrm{Adj}(M)}(\{H_{1},H_{2}\})=\check{\mathrm{sp}}^{\infty}(\{H_{1},H_{2}\}).
\]
 
\end{lem}

\begin{proof}
This is Lemma \ref{lem:cl-two} under $\beta$.
\end{proof}

\begin{cor}
\label{cor:Adj-rk2-flats}Let $H_{1}$ and $H_{2}$ be distinct mp-flats
of $M$ and put $F=H_{1}\cap H_{2}$. If $\mathrm{corank}_{M}(F)=2$,
then $\psi(F)$ is a rank-$2$ flat of $\mathrm{Adj}(M)$; moreover,
$\psi(F)=\left\{ H_{1},H_{2}\right\} $ if $H_{1}\cup H_{2}=E(M)$,
and $\psi(F)$ contains at least three members otherwise. If $\mathrm{corank}_{M}(F)\ge3$,
then $\left\{ H_{1},H_{2}\right\} $ is a rank-$2$ flat of $\mathrm{Adj}(M)$.
Every rank-$2$ flat of $\mathrm{Adj}(M)$ arises in one of these
ways.
\end{cor}

\begin{proof}
Put $C_{i}=\check{H}_{i}$ for $i=1,2$, so $C_{1}\cup C_{2}=\check{F}$,
and note that $\mathrm{corank}_{M}(F)\geq2$ since $H_{1}\neq H_{2}$.
Apply Corollary \ref{cor:rank2-flats} and Lemma \ref{lem:sp-two}
under $\beta$: the condition $\mathrm{nullity}_{M^{\ast}}(C_{1}\cup C_{2})=2$
becomes $\mathrm{corank}_{M}(F)=2$, the condition $C_{1}\cap C_{2}=\emptyset$
becomes $H_{1}\cup H_{2}=E(M)$, and $\beta$ carries $\mathcal{C}(M^{\ast}|\check{F})$
onto $\psi(F)$. 
\end{proof}

\begin{prop}
\label{prop:rk3-Adjoint}Let $M$ be a matroid of rank $3$, with
or without coloops. Then $\mathrm{Adj}(M)$ has rank $3$ and is an
adjoint of $M$. 
\end{prop}

\begin{proof}
Suppose first that $M$ has no coloops. Then $\mathrm{Adj}(M)$ has
rank $3$ by Corollaries \ref{cor:upper-bounds} and \ref{cor:lower-bound}.
We claim that $\psi(F)$ is a flat of $\mathrm{Adj}(M)$ for every
flat $F$ of $M$. By Proposition \ref{prop:Adj-simple}, this holds
for $F=E(M)$ and for mp-flats $F$, and it is clear for $F=\mathrm{cl}_{M}(\emptyset)$.
Now let $F$ be a rank-$1$ flat. Since $r(M)=3$, there are $x\notin F$
and $y\notin\mathrm{cl}_{M}(F\cup\{x\})$, and $H_{1}:=\mathrm{cl}_{M}(F\cup\{x\})$
and $H_{2}:=\mathrm{cl}_{M}(F\cup\{y\})$ are distinct members of
$\psi(F)$. Then $F\subseteq H_{1}\cap H_{2}\subsetneq H_{1}$ forces
$F=H_{1}\cap H_{2}$, which has corank $2$, so $\psi(F)$ is a rank-$2$
flat of $\mathrm{Adj}(M)$ by Corollary \ref{cor:Adj-rk2-flats}.
This proves the claim, so $\psi$ embeds the inverted flat lattice
$\mathcal{L}(M)'$ into $\mathcal{L}(\mathrm{Adj}(M))$. The restriction
of $\psi$ to the atoms of $\mathcal{L}(M)'$, that is, the mp-flats,
sends each $H$ to $\{H\}$, and these singletons are the atoms of
$\mathcal{L}(\mathrm{Adj}(M))$ by Proposition \ref{prop:Adj-simple}.
Thus $\mathrm{Adj}(M)$ is an adjoint of $M$. 

Suppose next that $M$ has coloops, and put $N=M\backslash\mathrm{cl}_{M^{\ast}}(\emptyset)$
and $p=\left|\mathrm{cl}_{M^{\ast}}(\emptyset)\right|$. Then $r(N)=3-p$,
and $\mathrm{Adj}(M)\simeq\mathrm{Adj}(N)\oplus U_{p,p}$ by the beginning
of this section. Moreover, $\mathcal{L}(M)\simeq\mathcal{L}(N)\times\mathcal{L}(U_{p,p})$,
and $\mathcal{L}(U_{p,p})$, the lattice of subsets of a $p$-element
set, is isomorphic to its inverted lattice via complementation. Hence,
taking the product of the map $\psi$ for $N$ with this isomorphism,
we reduce to showing that $\mathrm{Adj}(N)$ has rank $r(N)$ and
is an adjoint of $N$. This is immediate for $r(N)\le1$. If $r(N)=2$,
then the mp-flats of $N$ are its rank-$1$ flats, of which there
are at least two, so $\mathrm{Adj}(N)$ is simple on at least two
elements by Proposition \ref{prop:Adj-simple}, hence has rank at
least $2$, and rank exactly $2$ by Corollary \ref{cor:upper-bounds}.
Thus $\mathcal{L}(N)'$ and $\mathcal{L}(\mathrm{Adj}(N))$ are lattices
of rank $2$, and $\psi$ matches their atoms bijectively, so $\psi$
is an isomorphism between the two lattices.
\end{proof}

\begin{thm}
\label{thm:Adj-V8-ext}Let $M$ be a matroid of rank $4$ that has
$V_{8}$ as a restriction. Then $\mathrm{Adj}(M)$ is a connected
matroid of rank $3$, and $\check{\mathrm{sn}}(\mathcal{H}(M))=3$.
\end{thm}

\begin{proof}
We may assume that $M|[8]=V_{8}$ with $[8]\subseteq[n]=E(M)$, and
we label the circuits of $V_{8}$ as in Example \ref{exa:Vamos1}.
Since circuits and independent sets of $M|[8]$ are circuits and independent
sets of $M$, and since $r(M)=r(V_{8})=4$, the sets $C_{0},\dots,C_{4}$
are circuits of $M$ and $B=5678$ is a base of $M$.

The matroid $M$ has no coloops. Indeed, if $c$ were a coloop of
$M$, then $c\notin[8]$ as $V_{8}$ has no coloops, and $V_{8}$
would be a restriction of $M\backslash\{c\}$, so that $r(M)=r(M\backslash\{c\})+1\ge5$,
contradicting $r(M)=4$.

Loops of $M$ are coloops of $M^{\ast}$ and lie in no circuit of
$M^{\ast}$, so deleting $\mathrm{cl}_{M}(\emptyset)$ from $M$ changes
neither $\mathcal{C}(M^{\ast})$ nor $\mathrm{Adj}(M)$, up to the
identification of $\mathcal{H}(M\backslash\mathrm{cl}_{M}(\emptyset))$
with $\mathcal{H}(M)$ given by $H\mapsto H\cup\mathrm{cl}_{M}(\emptyset)$.
We may therefore assume that $M$ is loopless. 

Since $V_{8}$ is connected, the component of $M$ meeting $[8]$
includes $[8]$, so this component has rank $4$. Every other component
of $M$ then has rank $0$ and hence consists of a single loop. As
$M$ is loopless, $M$ is connected and so is $\mathrm{Adj}(M)$ by
Proposition \ref{prop:Adj-connected}. Since $r(\mathrm{Adj}(M))\ge3$
by Corollary \ref{cor:lower-bound}, it suffices by Corollary \ref{cor:upper-bounds}
to show that $\ensuremath{\check{\mathrm{sn}}(\mathcal{H}(M))=\mathrm{sn}(\mathcal{C}(M^{*}))\leq3}$. 

We show that each $\check{C}_{i}$ includes exactly one circuit of
$M^{*}$. Fix $i\in\left\{ 0,\dots,4\right\} $ and a three-element
subset $S$ of $C_{i}$. Since $\check{C}_{i}$ is an mp-flat of $M^{\ast}$,
we have $r_{M^{\ast}}(\check{C}_{i})=n-5<n-4=\left|\check{C}_{i}\right|$,
so $\check{C}_{i}$ is dependent in $M^{\ast}$. For each $e\in[8]\backslash C_{i}$,
the set $S\cup\{e\}$ is a base of $M$, since no other $C_{j}$ shares
three elements with $C_{i}$, and hence $\beta(S\cup\{e\})$ is a
base of $M^{*}$. As $\check{C}_{i}$ is dependent in $M^{\ast}$
and $\check{C}_{i}\subseteq\check{S}=\beta(S\cup\{e\})\cup\{e\}$,
the set $\check{C}_{i}$ includes exactly one circuit $D_{i}$ of
$M^{\ast}$, which is $C_{M^{\ast}}(e,\beta(S\cup\{e\}))$ for any
such $S$ and $e$. Hence $[8]\cap\check{C}_{i}\subseteq D_{i}$.

Let $\mathcal{A}=\{D_{0},D_{1},D_{4}\}$. We show that $\mathrm{sp}(\mathcal{A})$
includes the fundamental circuit system $\{C_{M^{*}}(e,\check{B}):e\in B\}$
of the base $\check{B}$ of $M^{*}$; by Propositions \ref{prop:span} and \ref{prop:sp-basic},
this implies that $\mathcal{A}$ Gaussian spans $\mathcal{C}(M^{*})$.
Since $5\in D_{0}\cap D_{4}$ and $(D_{0}\cup D_{4})\backslash\{5\}\subseteq(\check{C}_{0}\cup\check{C}_{4})\backslash\{5\}\subseteq\check{3457}\cup\{7\}$,
where $\check{3457}$ is a base of $M^{\ast}$, we have $\left\lfloor (D_{0}\cup D_{4})\backslash\{5\}\right\rfloor =C_{M^{\ast}}(7,\check{3457})=D_{3}$
(take $i=3$, $S=345$, and $e=7$ in the description of $D_{i}$
above). Similarly, $7\in D_{0}\cap D_{1}$ and $(D_{0}\cup D_{1})\backslash\{7\}\subseteq(\check{C}_{0}\cup\check{C}_{1})\backslash\{7\}\subseteq\check{1237}\cup\{3\}$
for the base $\check{1237}$ of $M^{*}$, so $\lfloor(D_{0}\cup D_{1})\backslash\{7\}\rfloor=C_{M^{*}}(3,\check{1237})=D_{2}$
(take $i=2$, $S=127$, and $e=3$). Since $D_{2}\cup D_{4}\subseteq\check{78}$
and $56\subseteq D_{2}\cap D_{4}$, we have $(D_{2}\cup D_{4})\backslash\{6\}\subseteq\check{B}\cup\{5\}$
and $(D_{2}\cup D_{4})\backslash\{5\}\subseteq\check{B}\cup\{6\}$,
so $\left\lfloor (D_{2}\cup D_{4})\backslash\{6\}\right\rfloor =C_{M^{\ast}}(5,\check{B})$
and $\left\lfloor (D_{2}\cup D_{4})\backslash\{5\}\right\rfloor =C_{M^{\ast}}(6,\check{B})$.
Similarly, $D_{1}\cup D_{3}\subseteq\check{56}$ and $78\subseteq D_{1}\cap D_{3}$
imply $\left\lfloor (D_{1}\cup D_{3})\backslash\{8\}\right\rfloor =C_{M^{\ast}}(7,\check{B})$
and $\left\lfloor (D_{1}\cup D_{3})\backslash\{7\}\right\rfloor =C_{M^{\ast}}(8,\check{B})$.

As in Example \ref{exa:Vamos1}(\ref{enu:Vamos-sn-3}), these eliminations
combine into executions of the Gaussian process for $\mathcal{A}$.
The pivot $D_{0}$ with $c_{1}=7$ replaces $D_{1}$ by $D_{2}$ and
leaves $D_{4}$ unchanged, since $D_{4}\subseteq\check{C}_{4}$ and
$7\in C_{4}$; a second step with the pivot $D_{4}$ and $c_{2}=6$
or $c_{2}=5$ gives $C_{M^{*}}(5,\check{B})$ or $C_{M^{*}}(6,\check{B})$,
respectively. Likewise, the pivot $D_{0}$ with $c_{1}=5$ replaces
$D_{4}$ by $D_{3}$ and leaves $D_{1}$ unchanged, since $D_{1}\subseteq\check{C}_{1}$
and $5\in C_{1}$; a second step with the pivot $D_{1}$ and $c_{2}=8$
or $c_{2}=7$ gives $C_{M^{*}}(7,\check{B})$ or $C_{M^{*}}(8,\check{B})$,
respectively. In each execution, $c_{1}$ and $c_{2}$ lie outside
the circuit to be obtained, since $C_{M^{*}}(e,\check{B})\subseteq\check{B}\cup\{e\}$,
$D_{2}\subseteq\check{C}_{2}$, and $D_{3}\subseteq\check{C}_{3}$.
The Gaussian processes for $(\mathcal{A},D_{2})$ and $(\mathcal{A},D_{3})$
succeed after the first steps above, so $D_{2},D_{3}\in\mathrm{sp}(\mathcal{A})$.
Thus $\{D_{0},\dots,D_{4}\}\cup\{C_{M^{\ast}}(e,\check{B}):e\in B\}\subseteq\mathrm{sp}(\mathcal{A})$,
so $\mathcal{C}(M^{\ast})=\mathrm{sp}^{2}(\mathcal{A})=\mathrm{sp}^{\infty}(\mathcal{A})$.
Therefore $\check{\mathrm{sn}}(\mathcal{H}(M))=\mathrm{sn}(\mathcal{C}(M^{\ast}))\le\left|\mathcal{A}\right|=3$.
\end{proof}

\begin{rem}
\label{rem:Dilworth-Adjoint}Let $M$ be as in Theorem \ref{thm:Adj-V8-ext}.
Since an adjoint of $M$ has the same rank as $M$, the matroid $\mathrm{Adj}(M)$
is not an adjoint of $M$, and the failure already appears at the
rank-$1$ flats of $M$. Let $F$ be a rank-$1$ flat of $M$. Take
rank-$2$ flats $F_{1}$ and $F_{2}$ of $M$ with $F_{1}\cap F_{2}=F$.
Both have corank $2$ in $M$, so $\psi(F_{1})$ and $\psi(F_{2})$
are rank-$2$ flats of $\mathrm{Adj}(M)$ by Corollary \ref{cor:Adj-rk2-flats},
and they are distinct since $\psi$ is injective. Then $\psi(F_{1})\cup\psi(F_{2})$
has rank at least $3$ in $\mathrm{Adj}(M)$, hence rank $3$ by Theorem
\ref{thm:Adj-V8-ext}, and so spans $\mathrm{Adj}(M)$. Since $\psi(F_{1})\cup\psi(F_{2})\subseteq\psi(F)$,
we have $\mathrm{cl}_{\mathrm{Adj}(M)}(\psi(F))=\mathcal{H}(M)$.
Yet $\psi(F)\neq\mathcal{H}(M)$ since $r_{M}(F)=1$, so $\psi(F)$
is not a flat of $\mathrm{Adj}(M)$. This is where $\mathrm{Adj}(M)$
departs from the complete adjoint, in which the image of every flat
is a flat \cite[Exercise 7.59]{Dilworth}. The complete adjoint of
$M$, into which the inverted flat lattice $\mathcal{L}(M)'$ embeds,
has rank at least $4$ and so does not coincide with $\mathrm{Adj}(M)$.
\end{rem}

\begin{problem}
Recall from Corollary \ref{cor:upper-bounds} that $r(\mathrm{Adj}(M))\le\check{\mathrm{sn}}(\mathcal{H}(M))\le k$.
When $k=3$, both inequalities are equalities by Proposition \ref{prop:rk3-Adjoint}.
When $M$ is as in Theorem \ref{thm:Adj-V8-ext}, we have $r(\mathrm{Adj}(M))=\check{\mathrm{sn}}(\mathcal{H}(M))=3<k$,
so only the second is strict. In addition, the first is an equality
whenever $M$ is connected and $\check{\mathrm{sn}}(\mathcal{H}(M))\leq3$,
by Corollary \ref{cor:sn-check-3-rank}. Characterize the matroids
$M$ for which $r(\mathrm{Adj}(M))=\check{\mathrm{sn}}(\mathcal{H}(M))$.
\end{problem}

\subsection*{Acknowledgements}

The author is grateful to Thomas Zaslavsky for helpful discussions
at an early stage of this research, and to Renzo Cavalieri for stimulating
conversations during the author's visit to Colorado State University.

\end{document}